\documentclass[11pt]{article}

\scrollmode

\usepackage{times}
\usepackage{amsmath,graphicx}
\usepackage{latexsym}
\usepackage{amscd,amsthm,amssymb}
\usepackage[all]{xy}
\usepackage[framemethod=tikz]{mdframed}

\newmdenv[
  roundcorner=5pt, backgroundcolor=gray!10, middlelinewidth=0, outerlinewidth=2, outerlinecolor=gray!70, skipabove=4pt, skipbelow=4pt, nobreak=true
]{greybox}

\newcommand{\tG}[1]{\ensuremath{\tilde{G}_{#1}}}

\newcommand{\varep}{\ensuremath{\varepsilon}}
\newcommand{\varepc}{\ensuremath{\varepsilon_{\mathbb{C}}}}
\newcommand{\mV}{\ensuremath{\mathcal{V}}}
\newcommand{\mH}{\ensuremath{\mathcal{H}}}

\newtheorem{thm}{Theorem}[section]
\newtheorem{prop}[thm]{Proposition}
\newtheorem{lem}[thm]{Lemma}
\newtheorem{cor}[thm]{Corollary}

\theoremstyle{definition}

\newtheorem{defn}[thm]{Definition}

\theoremstyle{remark}
\newtheorem{rem}[thm]{Remark}

\numberwithin{equation}{section}

\title{The field and Killing spinor equations of M-theory and type IIA/IIB supergravity in coordinate-free notation}
\author{M.~J.~D.~Hamilton\footnote{email: mark.hamilton@math.lmu.de}}
\date{\today}

\begin{document}

\maketitle

\begin{abstract} We review the actions of the supergravity theory in eleven dimensions as well as the type IIA and IIB supergravities in ten dimensions and derive the bosonic equations of motion in a coordinate-free notation. We also consider the existence of supersymmetries and the associated generalized Killing spinor equations. The aim of this note is to serve as a formulary and make the equations of supergravity more easily accessible to mathematicians.
\end{abstract}

\tableofcontents

\section{Introduction}

Supergravity theories in dimension ten and eleven are closely related to superstring and M-theory. The ``classical'' field equations (also known as equations of motion) of supergravity theories play a substantial role in research (in areas such as compactifications, $p$-branes or AdS/CFT-duality), even though string and M-theory are ultimately thought of as quantum theories.

It is not our intention to give an introduction into the physics of these theories, in particular, why the Lagrangians of the supergravity theories have this specific form or the meaning of supersymmetry (such an introduction and more details can be found, for example, in the textbooks listed in the references at the end of this article). The content of this note is also not original and appears in a number of different forms and places in the literature. Our aim rather is to make the field equations of supergravity theories more easily ``readable'' to a mathematical, in particular, differential-geometric audience. We also decided to include most of the proofs concerning the derivation of the equations of motion from the actions, so that they can be checked and possible mistakes we made can be spotted more easily.

We focus on the field equations for the bosonic fields, since in many applications the fermionic fields are set to zero. We also consider only the supergravity theory in eleven dimensions as well as the type IIA and IIB supergravities in ten dimensions and leave the type I and heterotic case to a future paper.

Recall that the classical Einstein field equation for the Lorentz metric on a four-dimensional spacetime $M$ in the presence of a source-free electromagnetic field, described by a $2$-form $F$, is given by \cite{MTW}, \cite{Wald}
\begin{align*}
\mathrm{Ric}(X,Y)-\frac{1}{2}g(X,Y)R&=8\pi T(X,Y)\\
&=2\left(\langle i_XF,i_YF\rangle-\frac{1}{2}g(X,Y)|F|^2\right),
\end{align*}
where $T$ is the energy-momentum tensor of the electromagnetic field and the equation should hold for all vector fields $X$ and $Y$ on the manifold $M$ (we work here in units where the gravitational constant is $1$). These are supplemented by the equations of motion for the electromagnetic field: the Maxwell equation
\begin{equation*}
d*F=0.
\end{equation*}
The field equations can be derived from the action
\begin{equation*}
\frac{1}{16\pi}\int_M\mathrm{dvol}_g\left(R-2|F|^2\right).
\end{equation*}
In addition there is a Bianchi identity if we assume that the field strength $F$ is only locally of the form $F=dA$, where $A$ is the electromagnetic potential:
\begin{equation*}
dF=0.
\end{equation*}
Similarly, the Einstein field equation in the presence of a massless real scalar field $\phi$ is given by
\begin{equation*}
\mathrm{Ric}(X,Y)-\frac{1}{2}g(X,Y)R=8\pi\left(d\phi(X)d\phi(Y)-\frac{1}{2}g(X,Y)|d\phi|^2\right)
\end{equation*}
and the equation of motion for $\phi$ is
\begin{equation*}
\Delta\phi=0.
\end{equation*}
Both can be derived from the action
\begin{equation*}
\frac{1}{16\pi}\int_M\mathrm{dvol}_g\left(R-8\pi|d\phi|^2\right).
\end{equation*}
It turns out that the bosonic field equations of supergravity theories are formally very similar. Besides the spacetime metric, the bosonic field content of supergravity theories involves certain differential forms (potentials and their associated field strengths) as well as a scalar field (smooth function) in dimension ten. There is again an Einstein equation with quadratic expressions of the scalar field and differential forms ``on the right hand side'', where in dimension four the energy-momentum tensor appears. These fields can thus be thought of as carrying energy that curves spacetime. 

In addition to the field equations for the metric, there are equations of motion for the scalar field and differential forms. The field equations for the differential forms, which we again call Maxwell equations, following \cite{FigS}, are formally similar to a coupled version of the classical Maxwell equations for the electromagnetic field. If we assume the potentials to exist only locally, we also have to add certain Bianchi identities. 

Given a solution to the equations of motion we can ask whether it is supersymmetric, i.e.~invariant under supersymmetries. Note that there is a difference between a Lagrangian or action of a field theory being invariant under certain symmetries and a solution to the field equations being invariant under symmetries. For example, the action of classical general relativity is invariant under diffeomorphisms and infinitesimal diffeomorphisms can be generated by vector fields. However, a spacetime metric which is a solution to the field equations will be invariant only under special diffeomorphisms (isometries), generated by Killing vector fields. Similarly, a supergravity action is invariant under local supersymmetries, generated by a spinor (a section of a spinor bundle over the manifold). However, a solution to the field equations will be invariant only under special supersymmetries, generated by so-called generalized Killing spinors. (We could also consider vector fields whose flow not only leaves the metric invariant, but also the other bosonic fields of a given supergravity solution. This leads to the notion of Killing superalgebras \cite{FigK}.)

Under a supersymmetry a boson transforms into a boson, hence the spinor parameter generating the supersymmetry has to be compensated by a fermion. Similarly the supersymmetry transformation of a fermion gives a fermion and involves the spinor parameter and the bosons. Since we have set the fermionic fields to zero, it follows that the bosonic fields in a solution of the field equations are automatically invariant under supersymmetries. It remains to show that the fermionic fields -- the gravitino and (in dimension 10) the dilatino -- are invariant under supersymmetries. For the gravitino, this reduces to the question of existence of a certain spinor (generating the supersymmetry) on the spacetime manifold that solves a generalized Killing spinor equation, called gravitino Killing spinor equation. Furthermore, the dilatino is invariant if this Killing spinor solves an additional algebraic equation, which we call dilatino Killing spinor equation (even though it does not contain a derivative of the spinor). Both equations involve the differential forms and the scalar field of the given background, i.e.~the solution to the equations of motion.

Since the gravitino and dilatino Killing spinor equation are linear, the space of solutions to these equations forms a real vector space (a vector subspace in the space of all spinors, i.e.~sections of the spinor bundle). Minkowski spacetime with the differential forms set to zero (and the dilaton constant in dimension ten) is a solution to the equations of motion. The gravitino Killing spinor equation in this case reduces to the equation for a parallel spinor and the dilatino Killing spinor equation is vacuous. The space of solutions thus has maximal real dimension, which is $32$ for both the eleven-dimensional and type IIA/IIB supergravity. In general, the vector space of supersymmetries of a given background will have smaller dimension (or dimension zero, i.e.~no supersymmetries at all).

In a final section we give a detailed derivation of the action and Killing spinor equations of type IIA supergravity on a ten-dimensional Lorentz manifold $N$ from the eleven-dimensional supergravity theory on the manifold $M=N\times S^1$. For the calculation of the curvature of the metric on $M$ it is very convenient to use the notion of Riemannian submersions and the O'Neill formulas.

\subsection{Conventions}
We use the signature $(-,+,\ldots,+)$ for the spacetime metric and the Einstein summation convention throughout.  We always assume that spacetimes are oriented.

\subsection{References}

Original references for supergravity theories are 
\begin{enumerate}
\item for the eleven-dimensional theory/M-theory \cite{CJS} and \cite{DLM}
\item for type IIA \cite{CW}, \cite{GiPe} and \cite{HN}
\item for type IIB \cite{Schwarz83}, \cite{SW} and \cite{HW}.
\end{enumerate} 
For the actions we follow the notation in the textbooks \cite{BBS}, \cite{BLT}, \cite{Johnson} and \cite{Polchinski} and the papers \cite{BHO}, \cite{Hull} and \cite{Schwarz95}. A general reference for supergravity theories is \cite{FP}. We give more specific references in each section. The idea for this work came from trying to rederive the fields equations given in coordinate-free form in \cite{FigH} and \cite{FigS}.

\subsection{Acknowledgements}
I would like to thank Mario Garcia Fernandez for informing me about an error in a previous version in the string frame Einstein equations, which also affected the dilaton equations. Thanks also to Mihaela Pilca for noticing that in a previous version the differential forms in Definition \ref{defn:11d supergravity} had a wrong degree.

\section{Differential forms and variations of fields}

Let $(M,g)$ be an oriented spacetime of dimension $(n+1)$, where $g$ is a Lorentz metric of signature $(-,+,\ldots,+)$. The canonical volume element in a positively oriented chart with local coordinates $x^\mu$ is given by
\begin{equation*}
\mathrm{dvol}_g=\sqrt{|g|}dx^0\wedge\ldots\wedge dx^{n},
\end{equation*}
where
\begin{equation*}
|g|=|\det(g_{\mu\nu})|=-\det(g_{\mu\nu})
\end{equation*}
is the absolute value of the determinant of the matrix with entries
\begin{equation*}
g_{\mu\nu}=g(\partial_\mu,\partial_\nu).
\end{equation*}
The Riemann curvature tensor, the Ricci curvature and the scalar curvature of $g$ are defined by
\begin{align*}
R(X,Y)Z&=\nabla_X\nabla_YZ-\nabla_Y\nabla_XZ-\nabla_{[X,Y]}Z\\
\mathrm{Ric}(X,Y)&=\mathrm{tr}(Z\longmapsto R(Z,X)Y)\\
R&=\mathrm{tr}(Z\longmapsto \mathrm{Ric}(Z)),
\end{align*}
where $\nabla$ is the Levi-Civita connection of $g$, $\mathrm{tr}$ denotes the trace and $\mathrm{Ric}(Z)$ is defined by
\begin{equation*}
g(\mathrm{Ric}(Z),X)=\mathrm{Ric}(Z,X)
\end{equation*}
for all vector fields $X$. We denote the components of the Ricci curvature $\mathrm{Ric}$ in the local coordinates $x^\mu$ by $R_{\mu\nu}$.

In the following results concerning variations of fields $\mathrm{td}$ denotes a total differential, whose integral over spacetime vanishes.
\begin{lem}\label{lem:var g}
We have under variations of the metric $g$
\begin{align*}
\delta R&=R_{\mu\nu}\delta g^{\mu\nu}+\mathrm{td}\\
\delta\mathrm{dvol}_g&=-\frac{1}{2}g_{\mu\nu}\mathrm{dvol}_g\delta g^{\mu\nu}.
\end{align*}
This implies for the variation of the Einstein-Hilbert Lagrangian
\begin{equation*}
\delta(R\mathrm{dvol}_g)=\left(R_{\mu\nu}-\frac{1}{2}g_{\mu\nu}R\right)\mathrm{dvol}_g\delta g^{\mu\nu}+\mathrm{td}.
\end{equation*}
\end{lem}
\begin{proof}
This result is well known and can be found in most textbooks on general relativity, such as \cite{MTW} and \cite{Wald}.
\end{proof}
We also need the following generalization:
\begin{lem}\label{lem:var Hilbert Einstein dilaton}
Let $\phi$ be a smooth function on $M$. Then under variations of the metric $g$ we have
\begin{align*}
\delta\left(e^{-2\phi}R\mathrm{dvol}_g\right)&=\mathrm{td}+\mathrm{dvol}_ge^{-2\phi}\biggl(R_{\mu\nu}-\frac{1}{2}g_{\mu\nu}R-4(\partial_\mu\phi)(\partial_\nu\phi)+2\nabla_{\mu}(\partial_\nu\phi)\\
&\hspace{3cm} +4(\partial_\alpha\phi)(\partial^{\alpha}\phi)g_{\mu\nu}-2\nabla_\alpha(\partial^{\alpha}\phi)g_{\mu\nu}\biggr)\delta g^{\mu\nu}.
\end{align*}
\end{lem}
\begin{proof}
According to equations (E.1.14) to (E.1.17) in \cite{Wald} we have
\begin{equation*}
\delta\left(R\mathrm{dvol}_g\right)=\mathrm{dvol}_g(\delta R_{\mu\nu})g^{\mu\nu}+\mathrm{dvol}_g\left(R_{\mu\nu}-\frac{1}{2}g_{\mu\nu}R\right)\delta g^{\mu\nu}
\end{equation*}
where
\begin{equation*}
(\delta R_{\mu\nu})g^{\mu\nu}=\nabla^{\alpha}\left(\nabla^{\beta}(\delta g_{\alpha\beta})-g^{\gamma\delta}\nabla_\alpha(\delta g_{\gamma\delta})\right).
\end{equation*}
We get
\begin{align*}
e^{-2\phi}(\delta R_{\mu\nu})g^{\mu\nu}&=\nabla^\alpha\left(e^{-2\phi}\left(\nabla^{\beta}(\delta g_{\alpha\beta})-g^{\gamma\delta}\nabla_\alpha(\delta g_{\gamma\delta})  \right)\right)\\
&\quad + 2e^{-2\phi}(\partial^{\alpha}\phi)\left(\nabla^{\beta}(\delta g_{\alpha\beta})-g^{\gamma\delta}\nabla_\alpha(\delta g_{\gamma\delta})\right)\\
&=\mathrm{td}+\nabla^{\beta}\left(2e^{-2\phi}(\partial^\alpha\phi)\delta g_{\alpha\beta}\right)+4e^{-2\phi}(\partial^\beta\phi)(\partial^\alpha\phi)\delta g_{\alpha\beta}\\
&\quad -2e^{-2\phi}\nabla^\beta(\partial^\alpha\phi)\delta g_{\alpha\beta}- \nabla_\alpha\left(2e^{-2\phi}(\partial^\alpha\phi)g^{\gamma\delta}\delta g_{\gamma\delta} \right)\\
&\quad -4e^{-2\phi}(\partial_\alpha\phi)(\partial^\alpha\phi)g^{\gamma\delta}\delta g_{\gamma\delta}+2e^{-2\phi}\nabla_\alpha(\partial^\alpha\phi)g^{\gamma\delta}\delta g_{\gamma\delta}\\
&=\mathrm{td}-e^{-2\phi}\biggl(4(\partial_\mu\phi)(\partial_\nu\phi)-2\nabla_{\mu}(\partial_\nu\phi)\\
&\hspace{2cm}-4(\partial_\alpha\phi)(\partial^{\alpha}\phi)g_{\mu\nu}+2\nabla_\alpha(\partial^{\alpha}\phi)g_{\mu\nu}\biggr)\delta g^{\mu\nu}.
\end{align*}
This implies the claim.
\end{proof}

\begin{defn}
We define the scalar product of two $k$-forms $F,G\in\Omega^k(M)$ by
\begin{align*}
\langle F,G\rangle &=\sum_{\mu_1<\ldots<\mu_k}F_{\mu_1\ldots\mu_k}G^{\mu_1\ldots\mu_k}\\
&=\frac{1}{k!}\sum_{\mu_1,\ldots,\mu_k}F_{\mu_1\ldots\mu_k}G^{\mu_1\ldots\mu_k},
\end{align*}
where
\begin{equation*}
F_{\mu_1\ldots\mu_k}=F(\partial_{\mu_1},\ldots,\partial_{\mu_k})
\end{equation*}
and we raise indices as usual with the inverse metric $g^{\mu\nu}$. The associated norm is given by
\begin{equation*}
|F|^2=\langle F,F\rangle.
\end{equation*}
We also define the Hodge star operator by
\begin{equation*}
\langle F,G\rangle
\mathrm{dvol}_g=F\wedge *G = G\wedge *F
\end{equation*}
\end{defn}
\begin{rem}
Note that the norm $|\cdot|$ on $k$-forms is not positive definite on a Lorentz manifold. In particular, $|F|^2=0$ does not imply $F=0$. Note also that our definition of the Hodge star operator does not necessarily coincide with the definition sometimes found in the literature on semi-Riemannian manifolds.
\end{rem}
If $e_0,e_1,\ldots,e_n$ are an oriented, orthonormal basis of vectors (also called vielbein) with
\begin{equation*}
g(e_0,e_0)=-1,\quad g(e_i,e_i)=1\quad\forall 1\leq i\leq n
\end{equation*}
and $\alpha^0,\alpha^1,\ldots,\alpha^n$ the dual basis of $1$-forms with $\alpha^i(e_j)=\delta^i_j$, then
\begin{equation*}
\mathrm{dvol}_g=\alpha^0\wedge\ldots\wedge\alpha^n
\end{equation*}
and
\begin{align*}
*\alpha^0&=-\alpha^1\wedge\ldots\wedge\alpha^n\\
*\alpha^i&=(-1)^i\alpha^0\wedge\ldots\widehat{\alpha^i}\wedge\ldots\alpha^n,
\end{align*}
where $\widehat{\alpha}^i$ means that this term is left out. More generally,
\begin{equation*}
*\left(\alpha^{m_0}\wedge\ldots\wedge\alpha^{m_k}\right)=\eta^{m_0m_0}\cdots\eta^{m_km_k}\epsilon_{m_0\ldots m_km_{k+1}\ldots m_n}\alpha^{m_{k+1}}\wedge\ldots\wedge\alpha^{m_n},
\end{equation*}
where on the right hand side there is no sum over the indices, $\{m_{k+1},\ldots,m_n\}$ is one complementary set to $\{m_0,\ldots,m_k\}$ and $\epsilon$ is totally antisymmetric with $\epsilon_{01\ldots n}=1$. For example,
\begin{equation*}
*(\alpha^0\wedge\alpha^1)=-\alpha^2\wedge\ldots\wedge\alpha^n.
\end{equation*}
\begin{lem}\label{lem:variation of form}
Let $G$ be a $k$-form. Then under a variation of the metric $g$ we have
\begin{equation*}
\frac{\delta |G|^2}{\delta g^{\mu\nu}}=\langle i_{\partial_\mu}G,i_{\partial_\nu}G\rangle,
\end{equation*}
where $i_XG$ denotes contraction with (insertion of) a vector $X$. This implies
\begin{equation*}
\frac{\delta}{\delta g^{\mu\nu}}(|G|^2\mathrm{dvol}_g)=\left(\langle i_{\partial_\mu}G,i_{\partial_\nu}G\rangle-\frac{1}{2}g_{\mu\nu}|G|^2\right)\mathrm{dvol}_g.
\end{equation*}
\end{lem}
\begin{proof}
We only have to prove the first claim. We write
\begin{equation*}
|G|^2=\frac{1}{k!}\sum_{\mu_1,\ldots,\mu_k,\nu_1,\ldots,\nu_k}G_{\mu_1\ldots\mu_k}G_{\nu_1\ldots\nu_k}g^{\mu_1\nu_1}\cdots g^{\mu_k\nu_k}.
\end{equation*}
Then
\begin{align*}
\frac{\delta |G|^2}{\delta g^{\mu\nu}}&=\frac{k}{k!}\sum_{\mu_2,\ldots,\mu_k,\nu_2,\ldots,\nu_k}G_{\mu\mu_2\ldots\mu_k}G_{\nu\nu_2\ldots\nu_k}g^{\mu_2\nu_2}\cdots g^{\mu_k\nu_k}\\
&=\frac{1}{(k-1)!}\sum_{\mu_2,\ldots,\mu_k}G_{\mu\mu_2\ldots\mu_k}G_{\nu}^{\,\,\,\mu_2\ldots\mu_k}\\
&=\langle i_{\partial_\mu}G,i_{\partial_\nu}G\rangle.
\end{align*}
\end{proof}
We also need the following lemma:
\begin{lem}
Let $G$ be a $k$-form. Then
\begin{equation*}
g^{\mu\nu}\langle i_{\partial_\mu}G,i_{\partial_\nu}G\rangle=k|G|^2
\end{equation*}
\end{lem}
\begin{proof}
We have
\begin{align*}
g^{\mu\nu}\langle i_{\partial_\mu}G,i_{\partial_\nu}G\rangle&=g^{\mu\nu}\frac{1}{(k-1)!}\sum_{\mu_2,\ldots,\mu_k}G_{\mu\mu_2\ldots\mu_k}G_{\nu}^{\,\,\,\mu_2\ldots\mu_k}\\
&=\frac{1}{(k-1)!}\sum_{\mu,\mu_2,\ldots,\mu_k}G_{\mu\mu_2\ldots\mu_k}G^{\mu\mu_2\ldots\mu_k}\\
&=\frac{k!}{(k-1)!}|G|^2\\
&=k|G|^2.
\end{align*}
\end{proof}
We will encounter variations of differential forms $C$, similar to the following.
\begin{lem}\label{lem:var k-form form C G=dC}
Let $C$ be a $k$-form and $G=dC$. Then under variations $\delta C$ we have
\begin{equation*}
\delta (G\wedge *G)=2(-1)^{k+1}\delta C\wedge d*G+\mathrm{td}.
\end{equation*}
\end{lem}
\begin{proof}
We have
\begin{align*}
\delta (dC\wedge *dC)&=d\delta C\wedge *dC+ dC\wedge *d\delta C\\
&=2d\delta C\wedge *dC\\
&=2d(\delta C\wedge *dC)-2(-1)^k\delta C\wedge d*dC.
\end{align*}
This implies the claim.
\end{proof}

The Laplacian $\Delta$ of a smooth function $\phi$ is defined by
\begin{equation*}
(\Delta \phi)\mathrm{dvol}_g= d*d\phi.
\end{equation*}
With our definition of the Hodge star we have
\begin{equation*}
*d\phi=(-1)^\mu\left(\sqrt{|g|}g^{\mu\nu}\partial_\nu\phi\right)dx^0\wedge\ldots\wedge\widehat{dx^\mu}\wedge\ldots\wedge dx^n
\end{equation*}
and
\begin{equation*}
\Delta\phi=\frac{1}{\sqrt{|g|}}\partial_\mu\left(\sqrt{|g|}g^{\mu\nu}\partial_\nu\phi\right).
\end{equation*}
This differs by a sign from the definition sometimes found in the literature.

The Hessian of a smooth function $\phi$ is defined by
\begin{equation*}
H^\phi(X,Y)=(\nabla_Xd\phi)(Y)=L_XL_Y\phi-L_{\nabla_XY}\phi.
\end{equation*}
In local coordinates
\begin{align*}
H^\phi(\partial_\mu,\partial_\nu)&=\nabla_\mu\partial_\nu\phi\\
\Delta\phi&=\nabla^\mu\partial_\mu\phi.
\end{align*}
We can then write the equation from Lemma \ref{lem:var Hilbert Einstein dilaton} as
\begin{align*}
\delta\left(e^{-2\phi}R\mathrm{dvol}_g\right)&=\mathrm{td}+\mathrm{dvol}_ge^{-2\phi}\biggl(R_{\mu\nu}-\frac{1}{2}g_{\mu\nu}R-4(\partial_\mu\phi)(\partial_\nu\phi)+2H^\phi(\partial_\mu,\partial_\nu)\\
&\hspace{3cm} +4|d\phi|^2g_{\mu\nu}-2\Delta\phi g_{\mu\nu}\biggr)\delta g^{\mu\nu}.
\end{align*}

\section{Spinors}

We recall some facts about spinors and fix our notation. References are \cite{Friedrich}, \cite{LM} and \cite{AFLV}.

\subsection{Linear algebra of spinors}\label{sect:lin alg spinors}

Let $\eta_{ab}$ be the standard diagonal $(-,+\ldots,+)$ Minkowski metric and $\Gamma_a$ for $a=0,\ldots,n$ (physical) gamma matrices with
\begin{equation*}
\{\Gamma_a,\Gamma_b\}=\Gamma_a\Gamma_b+\Gamma_b\Gamma_a=2\eta_{ab}\mathrm{Id}.
\end{equation*}
They are related to (mathematical) gamma matrices $\gamma_a$ by
\begin{equation*}
\gamma_a=i\Gamma_a,
\end{equation*}
so that
\begin{equation*}
\{\gamma_a,\gamma_b\}=\gamma_a\gamma_b+\gamma_b\gamma_a=-2\eta_{ab}\mathrm{Id}.
\end{equation*}
Indices of $\Gamma_a$ and $\gamma_a$ are raised with $\eta^{ab}$. We set
\begin{equation*}
\Gamma^{a_1\ldots a_k}=\frac{1}{k!}\sum_{\sigma\in S_k}\mathrm{sgn}(\sigma)\Gamma^{a_{\sigma(1)}}\cdots\Gamma^{a_{\sigma(k)}}.
\end{equation*}
Hence $\Gamma^{a_1\ldots a_k}$ is totally antisymmetric in the indices $a_1,\ldots,a_k$ and thus vanishes if two of them are the same. If all indices are distinct, we have
\begin{equation*}
\Gamma^{a_1\ldots a_k}=\Gamma^{a_1}\cdots\Gamma^{a_k}.
\end{equation*}
In particular
\begin{equation*}
\Gamma^{ab}=\frac{1}{2}\left[\Gamma^a,\Gamma^b\right].
\end{equation*}
If $\Sigma$ is a representation space of the Clifford algebra of $(\mathbb{R}^{n+1},\eta)$, then physical Clifford multiplication of a standard basis vector $e_a\in\mathbb{R}^{n+1}$ with a spinor $\psi\in\Sigma$ is given by
\begin{equation*}
e_a\cdot\psi=\Gamma_a\psi,
\end{equation*} 
related to mathematical Clifford multiplication by
\begin{equation*}
e_a\cdot\psi=(-i)\gamma_a\psi.
\end{equation*}
We define physical Clifford multiplication of a $k$-form $F\in\Lambda^k(\mathbb{R}^{n+1})^*$ with a spinor by
\begin{align*}
F\cdot\psi&=\frac{1}{k!}\sum_{a_1,\ldots, a_k}F_{a_1\ldots a_k}\Gamma^{a_1\ldots a_k}\psi\\
&=\sum_{a_1<\ldots<a_k}F_{a_1\ldots a_k}\Gamma^{a_1}\cdots\Gamma^{a_k}\psi,
\end{align*}
related to mathematical Clifford multiplication by
\begin{equation*}
F\cdot \psi=(-i)^k\sum_{a_1<\ldots<a_k}F_{a_1\ldots a_k}\gamma^{a_1}\cdots\gamma^{a_k}\psi.
\end{equation*}
 {\em We will use in the following only physical Clifford multiplications.} One can show that
\begin{equation*}
(\alpha\wedge F)\cdot\psi=\alpha\cdot(F\cdot\psi)-(i_{\alpha^\sharp}F)\cdot\psi,
\end{equation*}
where $\alpha^\sharp$ denotes the vector defined by $\alpha=\eta(\alpha^\sharp,-)$. Furthermore, if $F$ is a form of degree $k$, then
\begin{equation*}
X\cdot (F\cdot\psi)+(-1)^{k+1}F\cdot (X\cdot \psi)=(2i_XF)\cdot\psi
\end{equation*}
for all vectors $X$.

\subsection{Spinors on Lorentz manifolds}
We consider the orthochronous Lorentz group $SO^+(n,1)$ and assume that the frame bundle of the Lorentz manifold $(M,g)$ reduces to an $SO^+(n,1)$-principal bundle, denoted by $SO^+(M)$. Let $Spin(n,1)$ be the orthochronous spin group. Suppose that the manifold $(M,g)$ is spin, so that there exists a spin structure, which is a lift of the frame bundle $SO^+(M)\rightarrow M$ to a $Spin(n,1)$-principal bundle $Spin(M)\rightarrow M$ with $2$-fold covering
\begin{equation*}
\pi\colon Spin(M)\longrightarrow SO^+(M),
\end{equation*}
compatible with the Lie group actions and the $2$-fold covering 
\begin{equation*}
Spin(n,1)\rightarrow SO^+(n,1).
\end{equation*}
Let 
\begin{equation*}
\kappa\colon Spin(n,1)\rightarrow GL(\Delta)
\end{equation*}
be the complex (Dirac) spinor representation on $\Delta=\Delta_{n+1}$. Then the spinor bundle $S$ is defined as the associated vector bundle
\begin{equation*}
S=Spin(M)\times_\kappa \Delta.
\end{equation*}
A local oriented, orthochronous, orthonormal vielbein $e=(e_0,\ldots,e_n)$ on a contractible open subset $U\subset M$ determines a local trivialization of $SO^+(M)$ and two local trivializations $s^e_\pm$ of $Spin(M)$. In each of them we can write a local section $\Psi$ of $S$ (a spinor) as
\begin{equation*}
\Psi=[s^e_\pm,\psi_\pm],
\end{equation*} 
where $\psi_\pm$ is a map from $U$ to $\Delta$ and the square brackets denote the equivalence class in 
\begin{equation*}
S=(Spin(M)\times \Delta)/Spin(n,1).
\end{equation*}
We have $\psi_\pm=-\psi_\mp$. We choose one of the sections $s$ and the corresponding map $\psi$, so that
\begin{equation*}
\Psi=[s,\psi]
\end{equation*}
(the following notions are well-defined, independent of this choice). Clifford multiplication with a tangent vector 
\begin{equation*}
X=\sum X^a e_a \in TM
\end{equation*}
is given by
\begin{equation*}
X\cdot\Psi=\left[s,\sum X^a\Gamma_a\psi\right].
\end{equation*}
There is also a Clifford multiplication $F\cdot\Psi$ of differential forms $F\in\Omega^k(M)$ with spinors $\Psi$.

\subsection{The spin connection}
Let $\nabla$ denote the Levi-Civita connection on the tangent bundle $TM$ of the Lorentz metric $g$. In the local vielbein $e=(e_0,\ldots,e_n)$ we have
\begin{equation*}
\nabla e_a=\omega_{ab}\eta^{bc}\otimes e_c,
\end{equation*}
with certain $1$-forms $\omega_{ab}$. This induces a bundle connection 
\begin{equation*}
A_{SO}\in\Omega^1(SO^+(M),\mathfrak{so}(n,1))
\end{equation*} 
on $SO^+(M)$ and hence a bundle connection 
\begin{equation*}
A_{Spin}\in\Omega^1(Spin(M),\mathfrak{spin}(n,1))
\end{equation*}
on $Spin(M)$, since $\mathfrak{spin}(n,1)\cong \mathfrak{so}(n,1)$. We have associated local connection $1$-forms
\begin{equation*}
A_{Spin}^s=s^*A_{Spin}\in\Omega^1(U,\mathfrak{spin}(n,1)).
\end{equation*}
The connection on the $Spin$-principal bundle in turn defines a covariant derivative on the spinor bundle, locally given by
\begin{equation*}
\nabla_X\Psi=\left[s,\nabla_X\psi\right]
\end{equation*}
where 
\begin{equation*}
\nabla_X\psi=d\psi(X)+A_{Spin}^{s}(X)\cdot\psi
\end{equation*}
and $A_{Spin}^{s}(X)$ acts through the induced representation $\rho_*$ of $\mathfrak{spin}(n,1)$ on $\Delta$. We have the explicit formula
\begin{align*}
\nabla_X\psi&=d\psi(X)+\frac{1}{4}\omega_{ab}(X)\gamma^{ab}\psi\\
&=d\psi(X)-\frac{1}{4}\omega_{ab}(X)\Gamma^{ab}\psi.
\end{align*}
\begin{rem}
In the physics literature the forms $\omega_{ab}$ are often defined with the opposite sign and then the sign in front of $\omega_{ab}$ in the definition of the spin connection has to change as well. Typically in physics the anholonomy coefficients are defined by
\begin{equation*}
[e_a,e_b]=\Omega_{ab}^{\,\,\,\,\,\,c}e_c,
\end{equation*}
which can be calculated if we set
\begin{equation*}
e_a=E^\mu_a\partial_\mu.
\end{equation*}
Then the forms $\omega_{ab}$ are calculated by (this follows, for instance, from equations (8.19) and (8.20) in \cite{BBS})
\begin{equation*}
\omega_{cab}=\omega_{ab}(e_c)=\frac{1}{2}(-\Omega_{cab}+\Omega_{abc}-\Omega_{bca}).
\end{equation*}
This has the opposite sign of our definition (where $\omega_{cab}$ is determined by the Koszul formula).
\end{rem}

\subsection{Weyl spinors}\label{sect:Weyl spinors}
So far we only considered the general complex spinor representation on $\Delta$. If $d=n+1$ is the dimension of spacetime, then the complex dimension of $\Delta$ is 
\begin{itemize}
\item $\mathrm{dim}_{\mathbb{C}}\Delta=2^{(d-1)/2}$ if $d$ is odd
\item $\mathrm{dim}_{\mathbb{C}}\Delta=2^{d/2}$ if $d$ is even.
\end{itemize}
Spinors in $\Delta$ are also called Dirac spinors. In even dimensions $d=n+1=2k$ we define the chirality operator
\begin{equation*}
\Gamma_{n+2}=i^{1-k}\Gamma_0\Gamma_1\ldots\Gamma_n
\end{equation*}
(sometimes other factors instead of $i^{1-k}$ are used in the literature; the definition here differs by a sign from the one in \cite{Polchinski}). We have
\begin{equation*}
(\Gamma_{n+2})^2=\mathrm{Id},\quad \{\Gamma_{n+2},\Gamma^a\}=0,\quad [\Gamma_{n+2},\Gamma^{ab}]=0\quad\forall a,b=0,\ldots,n
\end{equation*}
(these properties hold independent of the conventions in the definition of $\Gamma_{n+2}$). We can split $\Delta$ into the $(\pm 1)$-eigenspaces of $\Gamma_{n+1}$:
\begin{equation*}
\Delta=\Delta^+\oplus\Delta^-.
\end{equation*}
Elements of $\Delta^+$ ($\Delta^-$) are called left- (right-)handed Weyl spinors. The complex dimension of $\Delta^\pm$ is $\frac{1}{2}2^{d/2}=2^{(d-2)/2}$.

This notion can be extended to spinor bundles over even-dimensional manifolds: Multiplication of a spinor with $\Gamma_{n+2}$ is given by
\begin{equation*}
\Gamma_{n+2}\cdot\psi=-i^{1-k}\mathrm{dvol}_g\cdot\psi,
\end{equation*}
where $\mathrm{dvol}_g$ is the canonical volume $(n+1)$-form. Multiplication of a spinor with $\Gamma_{n+2}$ is thus well-defined, independent of the choice of local oriented, orthonormal frame. We can split the spinor bundle $S$ into the $(\pm 1)$-eigenspaces of $\Gamma_{n+1}$:
\begin{equation*}
S=S^+\oplus S^-.
\end{equation*}
The Weyl spinor subbundles $S^\pm$ are preserved under the covariant derivative, while Clifford multiplication with a vector maps $S^\pm$ to $S^\mp$.

\subsection{Majorana spinors}
In certain dimensions $d$ the complex spinor space $\Delta$ admits a real structure, i.e.~a complex anti-linear, $Spin(n,1)$-equivariant map $\sigma\colon \Delta\rightarrow\Delta$ with $\sigma^2=\mathrm{Id}$. There is a real subspace
\begin{equation*}
\Delta^\sigma=\{\phi\in\Delta\mid\sigma(\phi)=\phi\}
\end{equation*}
and the $Spin(n,1)$-representation restricts to $\Delta^\sigma$. Then
\begin{equation*}
\Delta=\Delta^\sigma\oplus i\Delta^\sigma
\end{equation*}
and the real dimension of $\Delta^\sigma$ is half the real dimension of $\Delta$. The elements of $\Delta^\sigma$ are called Majorana spinors.

In some other dimensions the complex spinor space $\Delta$ admits a quaternionic structure, i.e.~a complex anti-linear, $Spin(n,1)$-equivariant map $J\colon\Delta\rightarrow\Delta$ with $J^2=-\mathrm{Id}$. Let $I\colon \Delta\rightarrow\Delta$ be multiplication with $i$. Then $I$, $J$ and $K=IJ$ define on $\Delta$ a $
Spin(n,1)$-equivariant structure of a quaternionic vector space. Elements of $\Delta$ are called symplectic Majorana spinors. These notions extend to spinor bundles, so that we can consider Majorana or symplectic Majorana spinors on manifolds.

Recall that Weyl spinors only exist in even dimensions. In certain even dimensions the real or quaternionic structure commutes with the projections onto the left- and right-handed Weyl spinors and thus these structures restrict to the Weyl spinor spaces. In these dimensions we can consider Majorana-Weyl and symplectic Majorana-Weyl spinors. 

For supersymmetry one is interested in real spinor representations of minimal dimensions. If they exist, Majorana or symplectic Majorana spinors are real representations of minimal dimensions. However, if in addition Majorana-Weyl or symplectic Majorana-Weyl spinors exist, then they are the real representations of minimal dimension. The number $N$ of supersymmetries of a given theory in dimension $d$ is the number of copies of these real spinor representations of minimal dimension.

The existence of real or quaternionic structures for the spinor representation depends for a metric of signature $(s,t)$ only on $s-t$. For a Lorentz metric of signature $(n,1)$ in dimension $d=n+1$ one can show that we have the following types of spinors, together with the minimal real dimension $D$ of a real spinor representation.
\begin{itemize}
\item $d\equiv 0,4\bmod 8$: There exist Weyl spinors and Majorana spinors, but there do not exist Majorana-Weyl spinors; $D=2^{d/2}$.
\item $d\equiv 1,3\bmod 8$: There exist Majorana spinors; $D=2^{(d-1)/2}$.
\item $d\equiv 2\bmod 8$: There exist Majorana-Weyl spinors; $D=2^{(d-2)/2}$.
\item $d\equiv 5,7\bmod 8$: There exist symplectic Majorana spinors; $D=2^{(d+1)/2}$
\item $d\equiv 6\bmod 8$: There exist symplectic Majorana-Weyl spinors; $D=2^{d/2}$.
\end{itemize}

\subsection{Spinors in dimension 10 and 11}\label{sect:spinors in dim 10 and 11}

In a  spacetime of dimension ten (with Lorentz signature) there exist both Weyl and Majorana spinors. In this dimension both structures are compatible, so that we can consider Majorana-Weyl spinors. The Weyl spinor spaces of each chirality have complex dimension $2^4=16$. The Majorana-Weyl spinor spaces of each chirality thus have real dimension $16$. 

With our convention for the chirality operator we have in dimension ten
\begin{equation*}
\Gamma_{11}=\Gamma_0\Gamma_1\cdots\Gamma_9.
\end{equation*}
This coincides with the convention used in equation (4.113) in \cite{BBS}. The chirality operator $\Gamma_{11}$ acts as $+1$ on left-handed and as $-1$ on right-handed spinors. 

In a spacetime of dimension eleven there do not exist Weyl spinors, but there do exist Majorana spinors. The complex dimension of the Dirac spinor space is $2^5=32$, which is equal to the real dimension of the Majorana spinor space.

In odd dimensional spacetime Clifford multiplication with a vector and with forms of odd degree depends on a choice of sign. We choose the sign in dimension eleven such that
\begin{equation*}
\Gamma_0\Gamma_1\cdots\Gamma_9\Gamma_{10}
\end{equation*}
acts as $+1$ on spinors (see appendix A in \cite{GaPa} and \cite[pp.~8-9]{FigS}). 

\section{Different versions of the actions and further remarks}

The supergravity actions sometimes appear in different versions in the literature. The main differences are:
\begin{enumerate}
\item Some authors use a different norm for $k$-forms $F$, that can be written as
\begin{equation*}
(F)^2=\sum_{\mu_1,\ldots,\mu_k}F_{\mu_1\ldots\mu_k}F^{\mu_1\ldots\mu_k},
\end{equation*}
where the sum extends over all $k$-tuples $\mu_1,\ldots,\mu_k$. We have
\begin{equation*}
|F|^2=\frac{1}{k!}(F)^2.
\end{equation*}
\item In supergravity theories in $10$ dimensions one can either use the so-called {\em string frame} or the {\em Einstein frame}. Both are related by a conformal change of the spacetime metric $g$. We will discuss both versions of the actions and the corresponding field equations.
\item The type IIB theory is sometimes written with a different set of fields (we call this the {\em symmetric notation}, because it has an apparent $SL(2,\mathbb{R})$-symmetry). We will also discuss this version of the action and field equations. 
\end{enumerate}
The bosonic field equations make sense on any oriented Lorentz manifold $M$ of dimension $10$ or $11$, respectively. However, for a supersymmetric theory with fermions, we need spin structures, hence we usually have to assume that $M$ is spin.

The actions involve certain constants $\kappa_{10}$ and $\kappa_{11}$ which are related to the Planck length and the gravitational coupling constant in ten and eleven dimensions (in ten dimensions the constant $\kappa_{10}$ is determined by the string length and independent of the string coupling, determined by the dilaton).

\section{M-Theory}

\subsection{The eleven-dimensional supergravity theory}

\begin{defn}\label{defn:11d supergravity}
The bosonic part of the action for the eleven-dimensional supergravity theory on a spacetime $M^{11}$ is given by
\begin{greybox}
\begin{equation*}
S=\frac{1}{2\kappa_{11}^2}\int_M \left(R-\frac{1}{2}|G|^2\right)\mathrm{dvol}_g-\frac{1}{12\kappa_{11}^2}\int_MC\wedge G\wedge G.
\end{equation*}
\end{greybox}
Here 
\begin{greybox}
\begin{enumerate}
\item $g$ is a Lorentz metric on $M$
\item $C$ is a $3$-form on $M$
\item $G$ is a $4$-form field strength defined by $G=dC$.
\end{enumerate}
\end{greybox}
\end{defn}
\begin{rem}
There is a different convention which replaces $C$ by $-C$, so that the action becomes
\begin{equation*}
S'=\frac{1}{2\kappa_{11}^2}\int_M \left(R-\frac{1}{2}|G|^2\right)\mathrm{dvol}_g+\frac{1}{12\kappa_{11}^2}\int_MC\wedge G\wedge G.
\end{equation*}
We use in the following the action $S$ and indicate later how the field equations change for the action $S'$.
\end{rem}
It is sometimes useful to split the action into an Einstein, Maxwell and Chern-Simons term:
\begin{greybox}
\begin{equation*}
S=S_E+S_M+S_{CS}
\end{equation*}
with
\begin{align*}
S_E&=\frac{1}{2\kappa_{11}^2}\int_M \mathrm{dvol}_g R\\
S_M&=-\frac{1}{4\kappa_{11}^2}\int_M \mathrm{dvol}_g |G|^2\\
S_{CS}&=-\frac{1}{12\kappa_{11}^2}\int_M C\wedge G\wedge G.
\end{align*}
\end{greybox}
\begin{thm}
The field equations that follow from the action are given by:
\begin{enumerate}
\item The {\bf Einstein equation:}
\begin{greybox}
\begin{equation*}
\mathrm{Ric}(X,Y)-\frac{1}{2}g(X,Y)R=\frac{1}{2}\langle i_XG,i_YG\rangle -\frac{1}{4}g(X,Y)|G|^2
\end{equation*}
\end{greybox}
or, equivalently, the {\bf Einstein equation in Ricci form:}
\begin{greybox}
\begin{equation*}
\mathrm{Ric}(X,Y)=\frac{1}{2}\langle i_XG,i_YG\rangle-\frac{1}{6}g(X,Y)|G|^2.
\end{equation*}
\end{greybox}
\item The {\bf Maxwell equation:}
\begin{greybox}
\begin{equation*}
d*G+\frac{1}{2}G\wedge G=0.
\end{equation*}
\end{greybox}
\item If we assume that the potential $C$ exists only locally, we add the {\bf Bianchi identity:}
\begin{greybox}
\begin{equation*}
dG=0.
\end{equation*}
\end{greybox}
\end{enumerate}
\end{thm}
\begin{rem}
Here and in the subsequent sections the Einstein equations should hold for all tangent vectors (or vector fields) $X$ and $Y$ on the manifold $M$.
\end{rem}
\begin{proof}
To prove the Einstein equation we consider variations of the metric. We then have to consider the variation of the term
\begin{equation*}
R\mathrm{dvol}_g-\frac{1}{2}|G|^2\mathrm{dvol}_g.
\end{equation*}
We get by Lemma \ref{lem:var g} and Lemma \ref{lem:variation of form}
\begin{align*}
\delta\left(R\mathrm{dvol}_g-\frac{1}{2}|G|^2\mathrm{dvol}_g\right)&=\biggl(R_{\mu\nu}-\frac{1}{2}g_{\mu\nu}R\\
&\quad -\frac{1}{2}\langle i_{\partial_\mu}G,i_{\partial_\nu}G\rangle+\frac{1}{4}g_{\mu\nu}|G|^2\biggr)\mathrm{dvol}_g\delta g^{\mu\nu} +\mathrm{td}.
\end{align*}
This implies the Einstein equation. To prove the Einstein equation in Ricci form, we calculate the scalar curvature:
\begin{align*}
-\frac{9}{2}R&=R-\frac{11}{2}R=g^{\mu\nu}\left(R_{\mu\nu}-\frac{1}{2}g_{\mu\nu}R\right)\\
&=\frac{1}{2}g^{\mu\nu}\langle i_{\partial_\mu}G,i_{\partial_\nu}G\rangle-\frac{11}{4}|G|^2\\
&=2|G|^2-\frac{11}{4}|G|^2\\
&=-\frac{3}{4}|G|^2.
\end{align*}
Hence
\begin{equation*}
R=\frac{1}{6}|G|^2,
\end{equation*}
which implies the Einstein equation in Ricci form.

To prove the Maxwell equation we consider variations of $C$. Then according to Lemma \ref{lem:var k-form form C G=dC}
\begin{equation*}
\delta(|G|^2\mathrm{dvol}_g)=\delta(G\wedge *G)=2\delta C\wedge d*G+\mathrm{td}.
\end{equation*}
We have
\begin{align*}
\delta(C\wedge G\wedge G)&=\delta C\wedge G\wedge G+C\wedge \delta dC\wedge G+C\wedge G\wedge\delta dC\\
&=\delta C\wedge G\wedge G+2C\wedge d\delta C\wedge G\\
&=\delta C\wedge G\wedge G-2d(C\wedge \delta C\wedge G)+2 dC\wedge \delta C\wedge G\\
&=d(\ldots)+3\delta C\wedge G\wedge G.
\end{align*}
This implies
\begin{equation*}
\delta\left(-\frac{1}{4\kappa_{11}^2}G\wedge *G-\frac{1}{12\kappa_{11}^2}C\wedge G\wedge G\right)=-\frac{1}{2\kappa_{11}^2}\delta C\wedge \left(d*G+\frac{1}{2}G\wedge G\right)+\mathrm{td}
\end{equation*}
und hence the Maxwell equation
\begin{equation*}
d*G+\frac{1}{2}G\wedge G=0.
\end{equation*}
\end{proof}

Along the way we saw: 
\begin{cor} The scalar curvature of a bosonic solution to the Einstein equation is given by
\begin{greybox}
\begin{equation*}
R=\frac{1}{6}|G|^2.
\end{equation*}
\end{greybox}
\end{cor}
\begin{rem}
If we use the action $S'$ with a plus sign in front of the Chern-Simons term, then the Einstein equation stays the same and the Maxwell equation becomes $d*G=\frac{1}{2}G\wedge G$.
\end{rem}

\subsection{Supersymmetry}\label{sect:M-theory susy}

In addition to the bosonic fields the eleven-dimensional supergravity theory contains one Majorana gravitino. We consider a Majorana spinor $\varep$. The following statement can be found, for example, in \cite{GaPa} and \cite{FigS}.
\begin{thm} A bosonic solution to the field equations of eleven-dimensional supergravity is supersymmetric if and only if there exists a non-zero Majorana spinor $\varep$ on $M$ that satisfies the {\bf gravitino Killing spinor equation}
\begin{greybox}
\begin{equation*}
\nabla_X\varep+\frac{1}{12}(X^\flat\wedge G-2i_XG)\cdot\varep=0,
\end{equation*}
where $X^\flat=g(X,\cdot)$ is the $1$-form associated to the vector field $X$.
\end{greybox}
\end{thm}
As before, $\nabla_X\varep$ denotes the covariant derivative, determined by the Levi-Civita connection, applied to $\varep$ and the equation is supposed to hold for all vector fields $X$. Note that Clifford multiplication in dimension $11$ involves the choice of a sign, see Section \ref{sect:spinors in dim 10 and 11}. The space of solutions to the gravitino Killing spinor equation has real dimension between $0$ and $32$.
 
Equivalent Killing spinor equations, stated in \cite{BBS} on page 306 and 312, respectively, are
\begin{greybox}
\begin{equation*}
\nabla_X\varep+\frac{1}{12}\left(X\cdot(G\cdot\varep)-3(i_X G)\cdot\varep\right)=0
\end{equation*}
\end{greybox}
and
\begin{greybox}
\begin{equation*}
\nabla_X\varep+\frac{1}{24}\left(3G\cdot(X\cdot\varep)-X\cdot(G\cdot\varep)\right)=0.
\end{equation*}
\end{greybox}
These equations follow with the identities mentioned in Section \ref{sect:lin alg spinors}.

\subsection{M-theory and anomaly cancellation}

The Lagrangian and Maxwell equation in M-theory receive corrections of higher order in the curvature of the metric, due to anomaly cancellation \cite{DLM}, \cite{BM}, \cite{GaPa}, \cite{FigS}. The action is modified by an additional summand
\begin{equation*}
-T_2\int_MC\wedge X_8,
\end{equation*}
known as the Green-Schwarz term, where $X_8$ is an $8$-form on $M$, of fourth order in the curvature $R$ of the metric $g$.
\begin{defn}The first order correction to the action in M-theory is given by
\begin{greybox}
\begin{equation*}
S_M=\frac{1}{2\kappa_{11}^2}\int_M \left(R\mathrm{dvol}_g-\frac{1}{2}|G|^2-\frac{1}{6}C\wedge G\wedge G\right)-T_2\int_MC\wedge X_8,
\end{equation*}
where
\begin{align*}
X_8&=\frac{1}{192}(p_1^2-4p_2)\\
T_2&=\left(\frac{2\pi^2}{\kappa_{11}^2}\right)^{1/3}
\end{align*}
and $p_1$ and $p_2$ are the first and second Pontryagin forms of the Riemann curvature $R$:
\begin{align*}
p_1&=-\frac{1}{8\pi^2}\mathrm{tr}\,\Omega^2\in\Omega^4(M)\\
p_2&=-\frac{1}{64\pi^4}\mathrm{tr}\,\Omega^4+\frac{1}{128\pi^4}(\mathrm{tr}\,\Omega^2)^2\in\Omega^8(M).
\end{align*}
\end{greybox}
\end{defn}
Here $\Omega$ is the local curvature form, a matrix of local $2$-forms $\Omega_{ab}$, defined on a Lorentz manifold of dimension $(n+1)$ by
\begin{equation*}
R(X,Y)e_a=\sum_{b=0}^{n}\Omega_{ab}(X,Y)e_b,
\end{equation*}
where $\{e_a\}$ is a local frame for the tangent bundle $TM$, for example, given by local coordinate vector fields $\{\partial_\mu\}$. In the matrix product $\Omega^k$ the $2$-form entries of $\Omega$ are multiplied with the wedge product (which is commutative on $2$-forms). Furthermore, $\mathrm{tr}$ denotes the standard trace as a matrix and we set
\begin{equation*}
(\mathrm{tr}\,\Omega^2)^2=(\mathrm{tr}\,\Omega^2)\wedge(\mathrm{tr}\,\Omega^2).
\end{equation*}
Even though the curvature form $\Omega$ is only defined locally, the Pontryagin forms are globally well-defined, independent of the choice of local frame.
\begin{cor}
The Maxwell equation changes in M-theory to
\begin{greybox}
\begin{equation*}
d*G+\frac{1}{2}G\wedge G=-\beta X_8,
\end{equation*}
where 
\begin{equation*}
\beta=2\kappa_{11}^2T_2=(4\pi\kappa_{11}^2)^{2/3}.
\end{equation*}
\end{greybox}
\end{cor}

\section{Supergravity type IIA}\label{sect: sugra IIA}

We now consider the $N=2$ supergravity theories in ten dimensions. We begin with the IIA theory.

\subsection{String frame}
\begin{defn}
The bosonic part of the action in the string frame for the ten-dimensional supergravity theory of type IIA on a spacetime $M^{10}$ is given by
\begin{greybox}
\begin{equation*}
S=S_{NS}+S_{RIIA}+S_{CSIIA}
\end{equation*}
with
\begin{align*}
S_{NS}&=\frac{1}{2\kappa_{10}^2}\int_M\mathrm{dvol}_ge^{-2\phi}\left(R+4|d\phi|^2-\frac{1}{2}|H_3|^2\right)\\
S_{RIIA}&=-\frac{1}{4\kappa_{10}^2}\int_M\mathrm{dvol}_g\left(|G_2|^2+|\tG{4}|^2\right)\\
S_{CSIIA}&=-\frac{1}{4\kappa_{10}^2}\int_M B_2\wedge G_4\wedge G_4.
\end{align*}
\end{greybox}
Here
\begin{greybox}
\begin{enumerate}
\item $g$ is a Lorentz metric on $M$
\item $\phi$ is a real scalar field (the dilaton)
\item $B_2$ is a $2$-form (the Neveu-Schwarz $B$-field)
\item $C_1$ and $C_3$ are a $1$-form and a $3$-form (Ramond-Ramond potentials).
\end{enumerate}
\end{greybox}
We also define the field strengths
\begin{greybox}
\begin{align*}
H_3&=dB_2\\
G_2&=dC_1\\
G_4&=dC_3
\end{align*}
and
\begin{equation*}
\tG{4}=G_4-C_1\wedge H_3.
\end{equation*}
\end{greybox}
\end{defn}
The field strengths $H_3,G_2,\tG{4}$ and the action are invariant under the generalized gauge transformations 
\begin{align*}
B_2'&=B_2+d\zeta_1\\
C_1'&=C_1+d\Lambda_0\\
C_3'&=C_3+d\Lambda_2+H_3\Lambda_0
\end{align*}
for differential forms $\zeta_1,\Lambda_0,\Lambda_2$.

\begin{thm} In the string frame the bosonic equations of motion of type IIA supergravity are given by:
\begin{enumerate}
\item {\bf Einstein equation:}
\begin{greybox}
\begin{align*}
\mathrm{Ric}(X,Y)-\frac{1}{2}g(X,Y)R&=-2H^\phi(X,Y)+2\left(\Delta\phi-|d\phi|^2\right)g(X,Y)\\
&\quad+\frac{1}{2}\left(\langle i_{X}H_3,i_{Y}H_3\rangle-\frac{1}{2}g(X,Y)|H_3|^2\right)\\
&\quad + \frac{1}{2}e^{2\phi}\left(\langle i_XG_2,i_YG_2\rangle-\frac{1}{2}g(X,Y)|G_2|^2\right)\\
&\quad +\frac{1}{2}e^{2\phi}\left(\langle i_{X}\tG{4},i_{Y}\tG{4}\rangle-\frac{1}{2}g(X,Y)|\tG{4}|^2\right)
\end{align*}
\end{greybox}
or, equivalently, the {\bf Einstein equation in Ricci form:}
\begin{greybox}
\begin{align*}
\mathrm{Ric}(X,Y)&=-2H^\phi(X,Y)-\frac{1}{4}\left(\Delta\phi-2|d\phi|^2 \right)g(X,Y)\\
&\quad +\frac{1}{2}\left(\langle i_XH_3,i_YH_3\rangle-\frac{1}{4}g(X,Y)|H_3|^2\right)\\
&\quad +\frac{1}{2}e^{2\phi}\left(\langle i_XG_2,i_YG_2\rangle-\frac{1}{8}g(X,Y)|G_2|^2\right)\\
&\quad  +\frac{1}{2}e^{2\phi}\left(\langle i_{X}\tG{4},i_{Y}\tG{4}\rangle-\frac{3}{8}g(X,Y)|\tG{4}|^2\right).
\end{align*}
\end{greybox}
\item {\bf Dilaton equation:}
\begin{greybox}
\begin{equation*}
\Delta\phi=2|d\phi|^2-\frac{1}{2}|H_3|^2+\frac{3}{4}e^{2\phi}|G_2|^2+\frac{1}{4}e^{2\phi}|\tG{4}|^2.
\end{equation*}
\end{greybox}
\item Using the dilaton equation we can write the Einstein equation in Ricci form as
\begin{greybox}
\begin{align*}
  \mathrm{Ric}(X,Y)&=-2H^\phi(X,Y)+\frac{1}{2}\langle i_XH_3,i_YH_3\rangle\\
&\quad +\frac{1}{2}e^{2\phi}\left(\langle i_XG_2,i_YG_2\rangle-\frac{1}{2}g(X,Y)|G_2|^2\right)\\
&\quad  +\frac{1}{2}e^{2\phi}\left(\langle i_{X}\tG{4},i_{Y}\tG{4}\rangle-\frac{1}{2}g(X,Y)|\tG{4}|^2\right).
\end{align*}
\end{greybox}
\item {\bf Maxwell equations:}
\begin{greybox}
\begin{align*}
d*G_2&=H_3\wedge *\tG{4}\\
d*\tG{4}&=-H_3\wedge \tG{4}\\
d\left(e^{-2\phi}*H_3\right)&=\frac{1}{2}\tG{4}\wedge\tG{4}-G_2\wedge*\tG{4}.
\end{align*}
\end{greybox}
\item If we assume that the potentials $B_2,C_1,C_3$ exist only locally, we add the {\bf Bianchi identities:}
\begin{greybox}
\begin{align*}
dH_3&=0\\
dG_2&=0\\
d\tG{4}&=H_3\wedge G_2.
\end{align*}
\end{greybox}
\end{enumerate}
\end{thm}
\begin{rem}
The string frame field equations where $G_2$ and $\tG{4}$ are set to zero are classical and can be found, for example, in \cite{CFMP}.
\end{rem}
We prove the Einstein equation.
\begin{proof}
Under variations of the metric $g$ we have by Lemma \ref{lem:var Hilbert Einstein dilaton}
\begin{align*}
\delta\left(e^{-2\phi}R\mathrm{dvol}_g\right)&=\mathrm{td}+\mathrm{dvol}_ge^{-2\phi}\biggl(R_{\mu\nu}-\frac{1}{2}g_{\mu\nu}R-4(\partial_\mu\phi)(\partial_\nu\phi)+2H^\phi(\partial_\mu,\partial_\nu)\\
&\hspace{3cm} +4|d\phi|^2g_{\mu\nu}-2\Delta\phi g_{\mu\nu}\biggr)\delta g^{\mu\nu}
\end{align*}
and by Lemma \ref{lem:variation of form}
\begin{align*}
\frac{\delta|d\phi|^2\mathrm{dvol}_g}{\delta g^{\mu\nu}}&=\left((\partial_\mu\phi)(\partial_\nu\phi)-\frac{1}{2}g_{\mu\nu}|d\phi|^2\right)\mathrm{dvol}_g\\
\frac{\delta |H_3|^2\mathrm{dvol}_g}{\delta g^{\mu\nu}}&=\left(\langle i_{\partial_\mu}H_3,i_{\partial_\nu}H_3\rangle-\frac{1}{2}g_{\mu\nu}|H_3|^2\right)\mathrm{dvol}_g\\
\frac{\delta |G_2|^2\mathrm{dvol}_g}{\delta g^{\mu\nu}}&=\left(\langle i_{\partial_\mu}G_2,i_{\partial_\nu}G_2\rangle-\frac{1}{2}g_{\mu\nu}|G_2|^2\right)\mathrm{dvol}_g\\
\frac{\delta |\tG{4}|^2\mathrm{dvol}_g}{\delta g^{\mu\nu}}&=\left(\langle i_{\partial_\mu}\tG{4},i_{\partial_\nu}\tG{4}\rangle-\frac{1}{2}g_{\mu\nu}|\tG{4}|^2\right)\mathrm{dvol}_g.
\end{align*}
This implies the Einstein equation. We now calculate the scalar curvature:
\begin{align*}
-4R&=R-5R=g^{\mu\nu}\left(R_{\mu\nu}-\frac{1}{2}R\right)\\
&=18\Delta\phi-20|d\phi|^2+\frac{1}{2}\left(3|H_3|^2-5|H_3|^2\right)\\
&\quad + \frac{1}{2}e^{2\phi}\left(2|G_2|^2-5|G_2|^2\right)+\frac{1}{2}e^{2\phi}\left(4|\tG{4}|^2-5|\tG{4}|^2\right)\\
&=18\Delta\phi-20|d\phi|^2-|H_3|^2-\frac{3}{2}e^{2\phi}|G_2|^2-\frac{1}{2}e^{2\phi}|\tG{4}|^2.
\end{align*}
This implies
\begin{equation*}
R=-\frac{9}{2}\Delta\phi+5|d\phi|^2+\frac{1}{4}|H_3|^2+\frac{3}{8}e^{2\phi}|G_2|^2+\frac{1}{8}e^{2\phi}|\tG{4}|^2
\end{equation*}
and hence the Einstein equation in Ricci form.
\end{proof}
We prove the dilaton equation:
\begin{proof}
Under variations of $\phi$ we have
\begin{align*}
\delta\left(\mathrm{dvol}_ge^{-2\phi}\left(R+4|d\phi|^2-\frac{1}{2}|H_3|^2\right)  \right)&=\delta (e^{-2\phi})\left(\mathrm{dvol}_g\left(R+4|d\phi|^2-\frac{1}{2}|H_3|^2\right)\right)\\
&\quad+4e^{-2\phi}\delta (d\phi\wedge*d\phi)\\
&=-2e^{-2\phi}\delta\phi\left(\mathrm{dvol}_g\left(R+4|d\phi|^2-\frac{1}{2}|H_3|^2\right)\right)\\
&\quad + 8e^{-2\phi}d\delta\phi\wedge*d\phi.
\end{align*}
We calculate
\begin{align*}
8e^{-2\phi}d\delta\phi\wedge*d\phi&=8d\left(\delta\phi e^{-2\phi}*d\phi\right)-8\delta\phi d\left(e^{-2\phi}*d\phi\right).
\end{align*}
Ignoring the total differential we get that the variation vanishes if and only if
\begin{equation*}
4d\left(e^{-2\phi}*d\phi\right)=-e^{-2\phi}\left(R+4|d\phi|^2-\frac{1}{2}|H_3|^2\right)\mathrm{dvol}_g.
\end{equation*}
Hence with the formula for the scalar curvature
\begin{align*}
d\left(e^{-2\phi}*d\phi\right)&=-\frac{1}{4}e^{-2\phi}\left(R+4|d\phi|^2-\frac{1}{2}|H_3|^2\right)\mathrm{dvol}_g\\
&=-\frac{1}{4}e^{-2\phi}\left(-\frac{9}{2}\Delta\phi+9|d\phi|^2-\frac{1}{4}|H_3|^2+\frac{3}{8}e^{2\phi}|G_2|^2+\frac{1}{8}e^{2\phi}|\tG{4}|^2\right)\mathrm{dvol}_g\\
&=\left(\frac{9}{8}e^{-2\phi}\Delta\phi-\frac{9}{4}e^{-2\phi}|d\phi|^2+\frac{1}{16}e^{-2\phi}|H_3|^2-\frac{3}{32}|G_2|^2-\frac{1}{32}|\tG{4}|^2\right)\mathrm{dvol}_g,
\end{align*}
Writing
\begin{equation*}
d\left(e^{-2\phi}*d\phi\right)=e^{-2\phi}\left(-2|d\phi|^2+\Delta\phi\right)\mathrm{dvol}_g
\end{equation*}
we get the dilaton equation.
\end{proof}
\begin{cor}
The scalar curvature of a bosonic solution to the IIA Einstein equation in the string frame is given by
\begin{greybox}
\begin{equation*}
R=-4|d\phi|^2+\frac{5}{2}|H_3|^2-3e^{2\phi}|G_2|^2-e^{2\phi}|\tG{4}|^2.
\end{equation*}
\end{greybox}
\end{cor}
We prove each Maxwell equation in a separate lemma.
\begin{lem} The Maxwell equation, obtained by varying $C_1$, is given by
\begin{equation*}
d*G_2=H_3\wedge *\tG{4}.
\end{equation*}
\end{lem}
\begin{proof}
We have under variations of $C_1$
\begin{align*}
\delta \left(G_2\wedge *G_2\right)&=2\delta dC_1\wedge *G_2\\
&=2d(\delta C_1 \wedge *G_2)+2\delta C_1\wedge d*G_2
\end{align*}
and
\begin{align*}
\delta \left(\tG{4}\wedge * \tG{4}\right)&=2\delta\tG{4}\wedge *\tG{4}\\
&=-2\delta C_1\wedge H_3\wedge *\tG{4}.
\end{align*}
This implies the equation.
\end{proof}
\begin{lem} The Maxwell equation, obtained by varying $C_3$, is given by
\begin{equation*}
d*\tG{4}=-H_3\wedge \tG{4}.
\end{equation*}
\end{lem}
\begin{proof}
We have under variations of $C_3$
\begin{align*}
\delta \left(\tG{4}\wedge *\tG{4}\right)&=2\delta dC_3\wedge *\tG{4}\\
&=2d(\delta C_3 \wedge *\tG{4})+2\delta C_3\wedge d*\tG{4}
\end{align*}
and
\begin{align*}
\delta \left(B_2\wedge G_4\wedge G_4\right)&=2B_2\wedge \delta dC_3\wedge G_4\\
&=2d(B_2\wedge \delta C_3\wedge G_4)-2H_3\wedge \delta C_3\wedge G_4\\
&=d(\ldots)+2\delta C_3\wedge H_3\wedge G_4\\
&=d(\ldots)+2\delta C_3\wedge H_3\wedge \tG{4}.
\end{align*}
This implies the equation.
\end{proof}

\begin{lem} The Maxwell equation, obtained by varying $B_2$, is given by
\begin{equation*}
d\left(e^{-2\phi}*H_3\right)=\frac{1}{2}\tG{4}\wedge\tG{4}-G_2\wedge*\tG{4}.
\end{equation*}
\end{lem}
\begin{proof}
We have under variations of $B_2$
\begin{align*}
e^{-2\phi}\delta(H_3\wedge *H_3)&=2 \delta dB_2\wedge e^{-2\phi}*H_3\\
&=2d\left(\delta B_2\wedge e^{-2\phi}*H_3\right)-2\delta B_2\wedge d\left(e^{-2\phi}*H_3\right).
\end{align*}
Furthermore,
\begin{align*}
\delta\left(\tG{4}\wedge *\tG{4}\right)&=2\delta\tG{4}\wedge *\tG{4}\\
&=-2C_1\wedge d\delta B_2\wedge*\tG{4}\\
&=2d\left(C_1\wedge \delta B_2\wedge *\tG{4}\right)-2G_2\wedge \delta B_2\wedge *\tG{4}+2C_1\wedge\delta B_2\wedge d*\tG{4}\\
&=d(\ldots)+2\delta B_2\wedge\left(-G_2\wedge *\tG{4}+C_1\wedge d*\tG{4}\right).
\end{align*}
Finally,
\begin{equation*}
\delta(B_2\wedge G_4\wedge G_4)=\delta B_2\wedge G_4\wedge G_4.
\end{equation*}
We get
\begin{equation*}
0=-2d\left(e^{-2\phi}*H_3\right)+2\left(-G_2\wedge *\tG{4}+C_1\wedge d*\tG{4}\right)+G_4\wedge G_4.
\end{equation*}
We have
\begin{equation*}
G_4\wedge G_4= \tG{4}\wedge\tG{4}+2 C_1\wedge H_3\wedge \tG{4}
\end{equation*}
and with the Maxwell equation for $\tG{4}$
\begin{equation*}
2C_1\wedge d*\tG{4}=-2C_1\wedge H_3\wedge \tG{4}.
\end{equation*}
Hence
\begin{equation*}
0=-2d\left(e^{-2\phi}*H_3\right)-2 G_2\wedge*\tG{4}+\tG{4}\wedge\tG{4}.
\end{equation*}
This implies the claim.
\end{proof}

Finally, concerning the Bianchi equations, it is easy to see that the last equation is equivalent to $G_4$ being closed.

\subsection{Supersymmetry in the string frame}

In addition to the bosonic fields the type IIA supergravity contains two Majorana-Weyl gravitinos of opposite chirality and two Majorana-Weyl dilatinos of opposite chirality. We combine them into one Majorana gravitino and one Majorana dilatino. We consider a Majorana spinor $\varep$. The following Killing spinor equations are derived in Section \ref{chapt:dim reduct}.
\begin{thm} A bosonic solution to the IIA field equations is supersymmetric if and only if there exists a non-zero Majorana spinor $\varep$ solving the following two equations:
\begin{enumerate}
\item {\bf Gravitino Killing spinor equation:}
\begin{greybox}
\begin{equation*}
0=\nabla_X\varep-\frac{1}{4}(i_XH_3)\cdot \Gamma_{11}\varep+\frac{1}{8}e^\phi\left(\tG{4}\cdot(X\cdot\varep)+G_2\cdot(X\cdot\Gamma_{11}\varep)\right).
\end{equation*}
\end{greybox}
\item {\bf Dilatino Killing spinor equation:}
\begin{greybox}
\begin{equation*}
0=\left(\frac{1}{4}e^{\phi}G_2-\frac{1}{3}d\phi\Gamma_{11}+\frac{1}{12}e^{\phi}\tG{4}\Gamma_{11}+\frac{1}{6}H_3\right)\cdot\varep.
\end{equation*}
\end{greybox}
\end{enumerate}
\end{thm}
The dilatino Killing spinor equation is purely algebraic and does not involve derivatives of the spinor. 
\begin{rem}
Both equations differ slightly from the equations stated in \cite{BBS} (see the comments at the end of Section \ref{chapt:dim reduct}). Our equations are equivalent to equations (2.11) and (2.12) in \cite{BRKOR} and to the equations in \cite[Section 2.1]{GPS} with $G^{(0)}$ and $\tilde{S}$ set to zero, respectively, as well as $\Gamma_{11}$ replaced by $-\Gamma_{11}$ and $C_1$ replaced by $-C_1$ (the choice of sign for $\Gamma_{11}$ is apparent in equation (A.6) in \cite{BKORV} and for $C_1$ in equation (2.22) in \cite{BRKOR}).
\end{rem}

\subsection{Einstein frame}\label{sect:IIA Einstein frame}

So far we considered the metric $g$ in the string frame. It is sometimes convenient to use the metric $g_E$ in Einstein frame, so that the Einstein-Hilbert part of the action has the standard form. We set
\begin{equation*}
g_E=e^{-\phi/2}g.
\end{equation*}
Then the scalar curvature of $g_E$ is given by
\begin{align*}
R_E&=e^{\phi/2}\left(R+\frac{1}{2}(n-1)\Delta\phi-\frac{1}{16}(n-1)(n-2)|d\phi|^2\right)\\
&=e^{\phi/2}\left(R+\frac{9}{2}\Delta\phi-\frac{9}{2}|d\phi|^2\right),
\end{align*}
where $\Delta\phi$ and the $|d\phi|^2$ are defined by the string frame metric $g$ and we use our sign convention for the Laplacian (see, for instance, \cite[Appendix D]{Wald} for the formula in the Lorentz case or \cite{YanoObata} for the formula in the Riemannian case, which also works on Lorentz manifolds).

In addition
\begin{equation*}
\mathrm{dvol}_{g_E}=e^{-5\phi/2}\mathrm{dvol}_g.
\end{equation*}
The norm of a $k$-form $F$ is then given by
\begin{equation*}
|F|^2_E=e^{k\phi/2}|F|^2,
\end{equation*}
hence
\begin{equation*}
\mathrm{dvol}_g|F|^2=e^{(5-k)\phi/2}\mathrm{dvol}_{g_E}|F|^2_E.
\end{equation*}
We also get for the Hodge star on $k$-forms $F$
\begin{equation*}
*F=e^{(5-k)\phi/2}*_EF.
\end{equation*}
\begin{lem}
We have
\begin{equation*}
\mathrm{dvol}_ge^{-2\phi}(R+4|d\phi|^2)=\mathrm{dvol}_{g_E}\left(R_E-\frac{1}{2}|d\phi|^2_E\right)+\mathrm{td}.
\end{equation*}
\end{lem}
\begin{proof}
We have
\begin{align*}
\mathrm{dvol}_{g_E}\left(R_E-\frac{1}{2}|d\phi|^2_E\right)&=\mathrm{dvol}_ge^{-2\phi}\left(R+\frac{9}{2}\Delta\phi+\frac{9}{2}|d\phi|^2-\frac{1}{2}|d\phi|^2 \right)\\
&=\mathrm{dvol}_ge^{-2\phi}\left(R+\frac{9}{2}\Delta\phi-5|d\phi|^2 \right).
\end{align*}
But
\begin{align*}
\frac{9}{2}\mathrm{dvol}_ge^{-2\phi}\Delta\phi&=\frac{9}{2}e^{-2\phi}d*d\phi\\
&=d\left(\frac{9}{2}e^{-2\phi}*d\phi\right)+9e^{-2\phi}d\phi\wedge*d\phi\\
&=d(\ldots)+9e^{-2\phi}|d\phi|^2\mathrm{dvol}_g.
\end{align*}
This implies the claim.
\end{proof}
\begin{thm}
In the Einstein frame the bosonic part of the type IIA action is given by 
\begin{greybox}
\begin{equation*}
S=S^E_{NS}+S^E_{RIIA}+S_{CSIIA},
\end{equation*}
where
\begin{align*}
S^E_{NS}&=\frac{1}{2\kappa_{10}^2}\int_M\mathrm{dvol}_{g_E}\left(R_E-\frac{1}{2}|d\phi|^2_E-\frac{1}{2}e^{-\phi}[H_3|_E^2\right)\\
S^E_{RIIA}&=-\frac{1}{4\kappa_{10}^2}\int_M\mathrm{dvol}_{g_E}\left(e^{3\phi/2}|G_2|_E^2+e^{\phi/2}|\tG{4}|_E^2\right)
\end{align*}
and $S_{CSIIA}$ is the same as before.
\end{greybox}
\end{thm}

The equations of motion change accordingly.

\begin{thm} In the Einstein frame the bosonic equations of motion of type IIA supergravity are given by:
\begin{enumerate}
\item {\bf Einstein equation:}
\begin{greybox}
\begin{align*}
\mathrm{Ric}_{E}(X,Y)-\frac{1}{2}g_{E}(X,Y)R_E&=\frac{1}{2}\left(d\phi(X)d\phi(Y)-\frac{1}{2}g_E(X,Y)|d\phi|_E^2\right)\\
&\quad+\frac{1}{2}e^{-\phi}\left(\langle i_{X}H_3,i_{Y}H_3\rangle_E-\frac{1}{2}g_E(X,Y)|H_3|_E^2\right)\\
&\quad + \frac{1}{2}e^{3\phi/2}\left(\langle i_XG_2,i_YG_2\rangle_E-\frac{1}{2}g_E(X,Y)|G_2|_E^2\right)\\
&\quad +\frac{1}{2}e^{\phi/2}\left(\langle i_{X}\tG{4},i_{Y}\tG{4}\rangle_E-\frac{1}{2}g_E(X,Y)|\tG{4}|_E^2\right)
\end{align*}
\end{greybox}
or, equivalently, the {\bf Einstein equation in Ricci form:}
\begin{greybox}
\begin{align*}
\mathrm{Ric}_{E}(X,Y)&=\frac{1}{2}d\phi(X)d\phi(Y)\\
&\quad+\frac{1}{2}e^{-\phi}\left(\langle i_{X}H_3,i_{Y}H_3\rangle_E-\frac{1}{4}g_E(X,Y)|H_3|^2_E\right)\\
&\quad + \frac{1}{2}e^{3\phi/2}\left(\langle i_XG_2,i_YG_2\rangle_E-\frac{1}{8}g_E(X,Y)|G_2|_E^2\right)\\
&\quad +\frac{1}{2}e^{\phi/2}\left(\langle i_{X}\tG{4},i_{Y}\tG{4}\rangle_E-\frac{3}{8}g_E(X,Y)|\tG{4}|_E^2\right).
\end{align*}
\end{greybox}
\item {\bf Dilaton equation:}
\begin{greybox}
\begin{equation*}
\Delta_E\phi=-\frac{1}{2}e^{-\phi}|H_3|_E^2+\frac{3}{4}e^{3\phi/2}|G_2|_E^2+\frac{1}{4}e^{\phi/2}|\tG{4}|^2_E
\end{equation*}
\end{greybox}
\item {\bf Maxwell equations:}
\begin{greybox}
\begin{align*}
d\left(e^{3\phi/2}*_EG_2\right)&=e^{\phi/2}H_3\wedge *_E\tG{4}\\
d\left(e^{\phi/2}*_E\tG{4}\right)&=-H_3\wedge \tG{4}\\
d\left(e^{-\phi}*_EH_3\right)&=\frac{1}{2}\tG{4}\wedge\tG{4}-e^{\phi/2}G_2\wedge*_E\tG{4}.
\end{align*}
\end{greybox}
\item The {\bf Bianchi identities} stay the same:
\begin{greybox}
\begin{align*}
dH_3&=0\\
dG_2&=0\\
d\tG{4}&=H_3\wedge G_2.
\end{align*}
\end{greybox}
\end{enumerate}
\end{thm}
\begin{rem}
The Einstein frame field equations, where all differential forms except one are set to zero, can be found, for example, in \cite[page 690]{BLT}.
\end{rem}
We only prove the dilaton equation, the remaining equations follow as before.
\begin{proof}
We have under variations of $\phi$
\begin{align*}
\delta|d\phi|^2_E\mathrm{dvol}_{g_E}&=2d\delta\phi\wedge*_Ed\phi\\
&=d(\ldots)-2\delta\phi (d*_Ed\phi)\\
&=d(\ldots)-2\delta\phi(\Delta_E\phi)\mathrm{dvol}_{g_E}.
\end{align*}
Also
\begin{equation*}
\delta e^{\lambda\phi}=\lambda e^{\lambda\phi}\delta\phi.
\end{equation*}
This implies
\begin{equation*}
0=\Delta_E\phi+\frac{1}{2}e^{-\phi}|H_3|_E^2-\frac{3}{4}e^{3\phi/2}|G_2|_E^2-\frac{1}{4}e^{\phi/2}|\tG{4}|^2_E
\end{equation*}
and hence the dilaton equation.
\end{proof}
\begin{cor}
The scalar curvature of a bosonic solution to the IIA Einstein equation in the Einstein frame is given by
\begin{greybox}
\begin{equation*}
R_E=\frac{1}{2}|d\phi|_E^2+\frac{1}{4}e^{-\phi}|H_3|_E^2+\frac{3}{8}e^{3\phi/2}|G_2|_E^2+\frac{1}{8}e^{\phi/2}|\tG{4}|_E^2.
\end{equation*}
\end{greybox}
\end{cor}

\section{Supergravity type IIB}

\subsection{String frame}
\begin{defn}
The bosonic part of the action in the string frame for the ten-dimensional supergravity theory of type IIB on a spacetime $M^{10}$ is given by
\begin{greybox}
\begin{equation*}
S=S_{NS}+S_{RIIB}+S_{CSIIB}
\end{equation*}
with
\begin{align*}
S_{NS}&=\frac{1}{2\kappa_{10}^2}\int_M\mathrm{dvol}_ge^{-2\phi}\left(R+4|d\phi|^2-\frac{1}{2}|H_3|^2\right)\\
S_{RIIB}&=-\frac{1}{4\kappa_{10}^2}\int_M\mathrm{dvol}_g\left(|G_1|^2+|\tG{3}|^2+\frac{1}{2}|\tG{5}|^2\right)\\
S_{CSIIB}&=-\frac{1}{4\kappa_{10}^2}\int_M C_4\wedge H_3\wedge G_3.
\end{align*}
\end{greybox}
Here
\begin{greybox}
\begin{enumerate}
\item $g$ is a Lorentz metric on $M$
\item $\phi$ is a real scalar field (the dilaton)
\item $B_2$ is a $2$-form (the Neveu-Schwarz $B$-field)
\item $C_0$ is a scalar field (the Ramond-Ramond axion)
\item $C_2$ and $C_4$ are a $2$-form and a $4$-form (Ramond-Ramond potentials).
\end{enumerate}
\end{greybox}
We also define the field strengths
\begin{greybox}
\begin{align*}
H_3&=dB_2\\
G_1&=dC_0\\
G_3&=dC_2\\
G_5&=dC_4
\end{align*}
and
\begin{align*}
\tG{3}&=G_3-C_0H_3\\
\tG{5}&=G_5-\frac{1}{2}H_3\wedge C_2+\frac{1}{2}G_3\wedge B_2.
\end{align*}
\end{greybox}
In addition to the action one demands the self-duality equation
\begin{equation*}
*\tG{5}=\tG{5}.
\end{equation*}
\end{defn}
The self-duality condition implies $|\tG{5}|^2=0$, since
\begin{equation*}
|\tG{5}|^2\mathrm{dvol}_g=\tG{5}\wedge \tG{5}=0,
\end{equation*}
because the form is of odd degree. We assume from now on that the self-duality of $\tG{5}$ is added to the equations of motion.

The field strengths $H_3,G_1,\tG{3},\tG{5}$ and the action are invariant under the generalized gauge transformations \cite{BLT}
\begin{align*}
B_2'&=B_2+d\zeta_1\\
C_0'&=C_0\\
C_2'&=C_2+d\Lambda_1\\
C_4'&=C_4+d\Lambda_3-\frac{1}{2}H_3\wedge\Lambda_1+\frac{1}{2}G_3\wedge\zeta_1
\end{align*}
for differential forms $\zeta_1,\Lambda_1,\Lambda_3$.
\begin{thm} In the string frame the bosonic equations of motion of type IIB supergravity are given by:
\begin{enumerate}
\item {\bf Einstein equation:}
\begin{greybox}
\begin{align*}
\mathrm{Ric}(X,Y)-\frac{1}{2}g(X,Y)R&=-2H^\phi(X,Y)+2\left(\Delta\phi-|d\phi|^2\right)g(X,Y)\\
&\quad+\frac{1}{2}\left(\langle i_{X}H_3,i_{Y}H_3\rangle-\frac{1}{2}g(X,Y)|H_3|^2\right)\\
&\quad + \frac{1}{2}e^{2\phi}\left(G_1(X)G_1(Y)-\frac{1}{2}g(X,Y)|G_1|^2\right)\\
&\quad +\frac{1}{2}e^{2\phi}\left(\langle i_{X}\tG{3},i_{Y}\tG{3}\rangle-\frac{1}{2}g(X,Y)|\tG{3}|^2\right)\\
&\quad +\frac{1}{4}e^{2\phi}\langle i_{X}\tG{5},i_{Y}\tG{5}\rangle
\end{align*}
\end{greybox}
or, equivalently, the {\bf Einstein equation in Ricci form:}
\begin{greybox}
\begin{align*}
\mathrm{Ric}(X,Y)&=-2H^\phi(X,Y)-\frac{1}{4}\left(\Delta\phi-2|d\phi|^2\right)g(X,Y)\\
&\quad +\frac{1}{2}\left(\langle i_XH_3,i_YH_3\rangle-\frac{1}{4}g(X,Y)|H_3|^2\right)\\
&\quad + \frac{1}{2}e^{2\phi}G_1(X)G_1(Y)\\
&\quad +\frac{1}{2}e^{2\phi}\left(\langle i_X\tG{3},i_Y\tG{3}\rangle-\frac{1}{4}g(X,Y)|\tG{3}|^2\right)\\
&\quad +\frac{1}{4}e^{2\phi}\langle i_X\tG{5},i_Y\tG{5}\rangle.
\end{align*}
\end{greybox}
\item {\bf Dilaton equation:}
\begin{greybox}
\begin{equation*}
\Delta\phi=2|d\phi|^2-\frac{1}{2}|H_3|^2+e^{2\phi}|G_1|^2+\frac{1}{2}e^{2\phi}|\tG{3}|^2.
\end{equation*}
\end{greybox}
\item Using the dilaton equation we can write the Einstein equation in Ricci form as
\begin{greybox}
\begin{align*}
\mathrm{Ric}(X,Y)&=-2H^\phi(X,Y)+\frac{1}{2}\langle i_XH_3,i_YH_3\rangle\\
&\quad + \frac{1}{2}e^{2\phi}\left(G_1(X)G_1(Y)-\frac{1}{2}g(X,Y)|G_1|^2\right)\\
&\quad +\frac{1}{2}e^{2\phi}\left(\langle i_X\tG{3},i_Y\tG{3}\rangle-\frac{1}{2}g(X,Y)|\tG{3}|^2\right)\\
&\quad +\frac{1}{4}e^{2\phi}\langle i_X\tG{5},i_Y\tG{5}\rangle.
\end{align*}
\end{greybox}

\item {\bf Maxwell equations:}
\begin{greybox}
\begin{align*}
d*G_1&=-H_3\wedge *\tG{3}\\
d*\tG{5}&=H_3\wedge \tG{3}\\
d*\tG{3}&=-H_3\wedge \tG{5}\\
d\left(e^{-2\phi}*H_3\right)&=\tG{3}\wedge\tG{5}+G_1\wedge*\tG{3}.
\end{align*}
\end{greybox}
\item We add the {\bf self-duality equation:}
\begin{greybox}
\begin{equation*}
*\tG{5}=\tG{5}.
\end{equation*}
\end{greybox}
\item If we assume that the potentials $B_2,C_0,C_2,C_4$ exist only locally, we add the {\bf Bianchi identities:}
\begin{greybox}
\begin{align*}
dH_3&=0\\
dG_1&=0\\
d\tG{3}&=H_3\wedge G_1\\
d\tG{5}&=H_3\wedge\tG{3}.
\end{align*}
\end{greybox}
\end{enumerate}
\end{thm}
\begin{rem}
Some of these equations differ from the corresponding equations in \cite{FigH}.
\end{rem}
We prove the Einstein equation:
\begin{proof}
The Einstein equation follows from the variation of $g$ as before in the case of IIA. We calculate the scalar curvature:
\begin{align*}
-4R&=R-5R=g^{\mu\nu}\left(R_{\mu\nu}-\frac{1}{2}R\right)\\
&=18\Delta\phi-20|d\phi|^2+\frac{1}{2}\left(3|H_3|^2-5|H_3|^2\right)\\
&\quad + \frac{1}{2}e^{2\phi}\left(|G_1|^2-5|G_1|^2\right)+\frac{1}{2}e^{2\phi}\left(3|\tG{3}|^2-5|\tG{3}|^2\right)\\
&=18\Delta\phi-20|d\phi|^2-|H_3|^2-2e^{2\phi}|G_1|^2-e^{2\phi}|\tG{3}|^2.
\end{align*}
This implies
\begin{equation*}
R=-\frac{9}{2}\Delta\phi+5|d\phi|^2+\frac{1}{4}|H_3|^2+\frac{1}{2}e^{2\phi}|G_1|^2+\frac{1}{4}e^{2\phi}|\tG{3}|^2
\end{equation*}
and hence the Einstein equation in Ricci form.
\end{proof}
We now prove the dilaton equation:
\begin{proof}
Under variations of $\phi$ we have as before in the case of IIA
\begin{equation*}
4d\left(e^{-2\phi}*d\phi\right)=-e^{-2\phi}\left(R+4|d\phi|^2-\frac{1}{2}|H_3|^2\right)\mathrm{dvol}_g.
\end{equation*}
Hence with the formula for the scalar curvature
\begin{align*}
d\left(e^{-2\phi}*d\phi\right)&=-\frac{1}{4}e^{-2\phi}\left(R+4|d\phi|^2-\frac{1}{2}|H_3|^2\right)\mathrm{dvol}_g\\
&=-\frac{1}{4}e^{-2\phi}\left(-\frac{9}{2}\Delta\phi+9|d\phi|^2-\frac{1}{4}|H_3|^2+\frac{1}{2}e^{2\phi}|G_1|^2+\frac{1}{4}e^{2\phi}|\tG{3}|^2\right)\mathrm{dvol}_g\\
&=\left(\frac{9}{8}e^{-2\phi}\Delta\phi - \frac{9}{4}e^{-2\phi}|d\phi|^2+\frac{1}{16}e^{-2\phi}|H_3|^2-\frac{1}{8}|G_1|^2-\frac{1}{16}|\tG{3}|^2\right)\mathrm{dvol}_g.
\end{align*}
Writing
\begin{equation*}
d\left(e^{-2\phi}*d\phi\right)=e^{-2\phi}\left(-2|d\phi|^2+\Delta\phi\right)\mathrm{dvol}_g
\end{equation*}
this implies the dilaton equation.
\end{proof}
\begin{cor}
The scalar curvature of a bosonic solution to the IIB Einstein equation in the string frame is given by
\begin{greybox}
\begin{equation*}
R=-4|d\phi|^2+\frac{5}{2}|H_3|^2-4e^{2\phi}|G_1|^2-2e^{2\phi}|\tG{3}|^2.
\end{equation*}
\end{greybox}
\end{cor}
We prove the Maxwell equations one after another.
\begin{lem} The Maxwell equation, obtained by varying $C_0$, is given by
\begin{equation*}
d*G_1=-H_3\wedge *\tG{3}.
\end{equation*}
\end{lem}
\begin{proof}
We have under variations of $C_0$
\begin{align*}
\delta \left(G_1\wedge *G_1\right)&=2\delta dC_0\wedge *G_1\\
&=2d(\delta C_0*G_1)-2\delta C_0d*G_1
\end{align*}
and
\begin{align*}
\delta \left(\tG{3}\wedge * \tG{3}\right)&=2\delta\tG{3}\wedge *\tG{3}\\
&=-2\delta C_0H_3\wedge *\tG{3}.
\end{align*}
This implies the equation.
\end{proof}

\begin{lem} The Maxwell equation, obtained by varying $C_4$, is given by
\begin{equation*}
d*\tG{5}=H_3\wedge \tG{3}.
\end{equation*}
\end{lem}
\begin{proof}
We have under variations of $C_4$
\begin{align*}
\delta \left(\tG{5}\wedge *\tG{5}\right)&=2\delta \tG{5}\wedge *\tG{5}\\
&=2\delta dC_4\wedge *\tG{5}\\
&=2d\left(\delta C_4\wedge *\tG{5}\right)-2\delta C_4\wedge d*\tG{5}.
\end{align*}
This implies
\begin{equation*}
\delta\left(\frac{1}{2}\tG{5}\wedge *\tG{5}+C_4\wedge H_3\wedge G_3\right)=d(\ldots)+\delta C_4\wedge (-d*\tG{5}+H_3\wedge G_3)
\end{equation*}
and thus
\begin{equation*}
d*\tG{5}=H_3\wedge G_3=H_3\wedge \tG{3}
\end{equation*}
by the definition of $\tG{3}$.
\end{proof}

\begin{lem} The Maxwell equation, obtained by varying $C_2$, is given by
\begin{equation*}
d*\tG{3}=-H_3\wedge \tG{5}.
\end{equation*}
\end{lem}
\begin{proof}
We have under variations of $C_2$
\begin{align*}
\delta \left(\tG{3}\wedge *\tG{3}\right)&=2\delta \tG{3}\wedge *\tG{3}\\
&=2\delta dC_2\wedge *\tG{3}\\
&=2d\left(\delta C_2\wedge *\tG{3}\right)-2\delta C_2\wedge d*\tG{3}.
\end{align*}
Similarly,
\begin{align*}
\delta \left(\tG{5}\wedge *\tG{5}\right)&=2\delta \tG{5}\wedge *\tG{5}\\
&=(-H_3\wedge \delta C_2+d\delta C_2\wedge B_2)\wedge *\tG{5}\\
&=-\delta C_2\wedge (H_3\wedge *\tG{5})+d(\delta C_2\wedge B_2\wedge *\tG{5})-\delta C_2\wedge(H_3\wedge *\tG{5}+B_2\wedge d*\tG{5})\\
&=d(\ldots)-\delta C_2(2H_3\wedge *\tG{5}+B_2\wedge d*\tG{5}).
\end{align*}
Finally
\begin{align*}
\delta(C_4\wedge H_3\wedge G_3)&=C_4\wedge H_3\wedge d\delta C_2\\
&=-d(C_4\wedge H_3\wedge \delta C_2)+G_5\wedge H_3\wedge \delta C_2\\
&=d(\ldots)+\delta C_2\wedge G_5\wedge H_3.
\end{align*}
This implies
\begin{equation*}
0=2d*\tG{3}+H_3\wedge *\tG{5}+\frac{1}{2} B_2\wedge d*\tG{5}-G_5\wedge H_3.
\end{equation*}
But the Maxwell equation for $\tG{5}$ and its defintion imply
\begin{align*}
\frac{1}{2}B_2\wedge d*\tG{5}-G_5\wedge H_3&=\frac{1}{2}H_3\wedge G_3\wedge B_2+H_3\wedge G_5\\
&=H_3\wedge\left(\frac{1}{2}G_3\wedge B_2+G_5\right)\\
&=H_3\wedge \tG{5}.
\end{align*}
Adding self-duality of $\tG{5}$ we get
\begin{equation*}
0=2d*\tG{3}+2H_3\wedge \tG{5}.
\end{equation*}
This implies the Maxwell equation for $\tG{3}$.
\end{proof}

\begin{lem} The Maxwell equation, obtained by varying $B_2$, is given by
\begin{equation*}
d\left(e^{-2\phi}*H_3\right)=\tG{3}\wedge\tG{5}+G_1\wedge*\tG{3}.
\end{equation*}
\end{lem}
\begin{proof}
We have as before
\begin{align*}
e^{-2\phi}\delta(H_3\wedge *H_3)&=2 \delta dB_2\wedge e^{-2\phi}*H_3\\
&=2d\left(\delta B_2\wedge e^{-2\phi}*H_3\right)-2\delta B_2\wedge d\left(e^{-2\phi}*H_3\right).
\end{align*}
Furthermore,
\begin{align*}
\delta\left(\tG{3}\wedge *\tG{3}\right)&=2\delta\tG{3}\wedge *\tG{3}\\
&=-2C_0d\delta B_2\wedge*\tG{3}\\
&=-2d\left(C_0\delta B_2\wedge *\tG{3}\right)+2G_1\wedge \delta B_2\wedge *\tilde{G}_3+2C_0\delta B_2\wedge d*\tG{3}\\
&=d(\ldots)+2\delta B_2\wedge\left(G_1\wedge *\tG{3}+C_0d*\tG{3}\right)
\end{align*}
and
\begin{align*}
\delta\left(\tG{5}\wedge *\tG{5}\right)&=2\delta\tG{5}\wedge*\tG{5}\\
&=-d\delta B_2\wedge C_2\wedge*\tG{5}+G_3\wedge \delta B_2\wedge*\tG{5}\\
&=-d\left(\delta B_2\wedge C_2\wedge*\tG{5}\right)+\delta B_2\wedge G_3\wedge*\tG{5}\\
&\quad +\delta B_2\wedge C_2\wedge d*\tG{5}+G_3\wedge \delta B_2\wedge*\tG{5}\\
&=d(\ldots)+\delta B_2\wedge\left(2G_3\wedge*\tG{5}+ C_2\wedge d*\tG{5}\right).
\end{align*}
Finally,
\begin{align*}
\delta(C_4\wedge H_3\wedge G_3)&=C_4\wedge d\delta B_2\wedge G_3\\
&=d\left(C_4\wedge \delta B_2\wedge G_3\right)-G_5\wedge \delta B_2\wedge G_3\\
&=d(\ldots)-\delta B_2\wedge G_5\wedge G_3.
\end{align*}
We get
\begin{align*}
0&=-2d\left(e^{-2\phi}*H_3\right)+2\left(G_1\wedge *\tG{3}+C_0d*\tG{3}\right)\\
&\quad +\left(G_3\wedge*\tG{5}+ \frac{1}{2}C_2\wedge d*\tG{5}\right)- G_5\wedge G_3.
\end{align*}
Using the Maxwell equation for $\tG{5}$ we have
\begin{align*}
\frac{1}{2}C_2\wedge d*\tG{5}&=\frac{1}{2}C_2\wedge H_3\wedge \tG{3}\\
&=\frac{1}{2}C_2\wedge H_3\wedge G_3,
\end{align*}
so that
\begin{align*}
\frac{1}{2}C_2\wedge d*\tG{5}- G_5\wedge G_3&=\left(\frac{1}{2}C_2\wedge H_3-G_5\right)\wedge G_3\\
&=\left(\frac{1}{2}G_3\wedge B_2-\tG{5}\right)\wedge G_3\\
&=G_3\wedge\tG{5}.
\end{align*}
Using the Maxwell equation for $\tG{3}$ we have
\begin{align*}
2C_0d*\tG{3}+2G_3\wedge\tG{5}&=(-2C_0H_3+2G_3)\wedge\tG{5}\\
&=2\tG{3}\wedge\tG{5}.
\end{align*}
We get
\begin{equation*}
0=-2d\left(e^{-2\phi}*H_3\right)+2 G_1\wedge*\tG{3}+2\tG{3}\wedge\tG{5}.
\end{equation*}
This implies the claim.
\end{proof}
Finally, concerning the Bianchi identities, it is easy to see that the last two equations are equivalent to $G_3$ and $G_5$ being closed.

\subsection{Supersymmetry in the string frame}

In addition to the bosonic fields, the type IIB supergravity contains two left-handed Majorana-Weyl gravitinos and two right-handed Majorana-Weyl dilatinos. We consider a doublet of left-handed Majorana-Weyl spinors
\begin{equation*}
\varep=\left(\begin{array}{c}\varep_1\\ \varep_2  \end{array}\right)\in S_+\oplus S_+.
\end{equation*}
Recall that the Pauli matrices are given by
\begin{equation*}
\sigma_1=\left(\begin{array}{cc} 0 & 1 \\ 1 & 0  \end{array}\right),\quad \sigma_2=\left(\begin{array}{cc} 0 & -i \\ i & 0  \end{array}\right),\quad \sigma_3=\left(\begin{array}{cc} 1 & 0 \\ 0 & -1  \end{array}\right).
\end{equation*}
They act on spinor-doublets. We set
\begin{equation*}
\lambda_1=\sigma_1,\quad \lambda_2=i\sigma_2,\quad \lambda_3=\sigma_3,
\end{equation*}
which span $\mathfrak{sl}(2,\mathbb{R})$. We also define the following $\mathfrak{sl}(2,\mathbb{R})$ valued differential forms:
\begin{greybox}
\begin{align*}
\tilde{\Omega}&=\frac{1}{8}e^\phi\left(G_1\lambda_2-\tG{3}\lambda_1+\frac{1}{2}\tG{5}\lambda_2\right)\\
\Omega&=e^{\phi}\left(\frac{1}{2}\tG{3}\lambda_1-G_1\lambda_2\right).
\end{align*}
\end{greybox}
The following version of the Killing spinor equations can be found in \cite{FigK}, who refer to \cite{BRKR} and \cite{HS}.
\begin{thm}\label{Fig IIB Killing spinor} A bosonic solution to the IIB field equations is supersymmetric if and only if there exists a non-zero left-handed Majorana-Weyl spinor-doublet $\varep$ solving the following two equations:
\begin{enumerate}
\item {\bf Gravitino Killing spinor equation:}
\begin{greybox}
\begin{equation*}
0=\nabla_X\varep+\frac{1}{4}(i_XH_3)\lambda_3\cdot\varep+\tilde{\Omega}\cdot(X\cdot\varep).
\end{equation*}
\end{greybox}
\item {\bf Dilatino Killing spinor equation:}
\begin{greybox}
\begin{equation*}
0=\left(d\phi+\frac{1}{2}H_3\lambda_3+\Omega\right)\cdot\varep.
\end{equation*}
\end{greybox}
\end{enumerate}
\end{thm}
Note again that the dilatino Killing spinor equation does not involve derivatives of the spinor and is purely algebraic. The following version of the Killing spinor equations can be found in \cite{BBS}.
\begin{prop}Defining a complex left-handed Weyl spinor $\varepc=\varep_2+i\varep_1$ we can write the IIB Killing spinor equations equivalently as follows:
\begin{enumerate}
\item {\bf Gravitino Killing spinor equation:}
\begin{greybox}
\begin{align*}
0&=\nabla_X\varepc-\frac{1}{4}(i_XH_3)\cdot\varepc^*\\
&\quad+\frac{i}{8}e^{\phi}\left(G_1\cdot(X\cdot\varepc)-\tG{3}\cdot(X\cdot\varepc^*)+\frac{1}{2}\tG{5}\cdot(X\cdot\varepc)\right).
\end{align*}
\end{greybox}
\item {\bf Dilatino Killing spinor equation:}
\begin{greybox}
\begin{equation*}
0=d\phi\cdot\varepc-\frac{1}{2}H_3\cdot\varep_{\mathbb{C}}^*+ie^{\phi}\left(\frac{1}{2}\tG{3}\cdot \varepc^*-G_1\cdot\varepc\right).
\end{equation*}
\end{greybox}

\end{enumerate}
Here $\varepc^*=\varep_2-i\varep_1$ and the equations are supposed to hold for real and imaginary parts separately.
\end{prop}
\begin{proof}We check equivalence of both versions of the equations. The gravitino Killing spinor equations from Theorem \ref{Fig IIB Killing spinor} are:
\begin{align*}
0&=\nabla_X\varep_1+\frac{1}{4}(i_XH)\cdot\varep_1\\
&\quad + \frac{1}{8}e^{\phi}\left(G_1\cdot(X\cdot\varep_2)-\tG{3}\cdot(X\cdot\varep_2)+\frac{1}{2}\tG{5}\cdot(X\cdot\varep_2)\right)\\
0&=\nabla_X\varep_2-\frac{1}{4}(i_XH)\cdot\varep_2\\
&\quad + \frac{1}{8}e^{\phi}\left(-G_1\cdot(X\cdot\varep_1)-\tG{3}\cdot(X\cdot\varep_1)-\frac{1}{2}\tG{5}\cdot(X\cdot\varep_1)\right).
\end{align*}
Similarly, the dilatino Killing spinor equations from Theorem \ref{Fig IIB Killing spinor} are:
\begin{align*}
0&=d\phi\cdot\varep_1+\frac{1}{2}H_3\cdot\varep_1+e^\phi\left(\frac{1}{2}\tG{3}\cdot\varep_2-G_1\cdot\varep_2\right)\\
0&=d\phi\cdot\varep_2-\frac{1}{2}H_3\cdot\varep_2+e^\phi\left(\frac{1}{2}\tG{3}\cdot\varep_1+G_1\cdot\varep_1\right).
\end{align*}
These are equal to the imaginary and real parts of the gravitino and dilatino Killing equations written with the complex Weyl spinor $\varepc$.
\end{proof}

\subsection{Einstein frame}

Using our calculations in the case of IIA, we see that in the Einstein frame with
\begin{equation*}
g_E=e^{-\phi/2}g
\end{equation*}
we have:
\begin{thm}
In the Einstein frame the bosonic part of the type IIB action is given by 
\begin{greybox}
\begin{equation*}
S=S^E_{NS}+S^E_{RIIB}+S_{CSIIB},
\end{equation*}
where
\begin{align*}
S^E_{NS}&=\frac{1}{2\kappa_{10}^2}\int_M\mathrm{dvol}_{g_E}\left(R_E-\frac{1}{2}|d\phi|^2_E-\frac{1}{2}e^{-\phi}[H_3|_E^2\right)\\
S^E_{RIIB}&=-\frac{1}{4\kappa_{10}^2}\int_M\mathrm{dvol}_{g_E}\left(e^{2\phi}|G_1|_E^2+e^{\phi}|\tG{3}|_E^2+\frac{1}{2}|\tG{5}|_E^2\right)
\end{align*}
and $S_{CSIIB}$ is the same as before.
\end{greybox}
\end{thm}
For the field equations we get:
\begin{thm}\label{thm:Einstein IIB} In the Einstein frame the bosonic equations of motion of type IIB supergravity are given by:
\begin{enumerate}
\item {\bf Einstein equation:}
\begin{greybox}
\begin{align*}
\mathrm{Ric}_E(X,Y)-\frac{1}{2}g_E(X,Y)R_E&=\frac{1}{2}\left(d\phi(X)d\phi(Y)-\frac{1}{2}g_E(X,Y)|d\phi|_E^2\right)\\
&\quad+\frac{1}{2}e^{-\phi}\left(\langle i_{X}H_3,i_{Y}H_3\rangle_E-\frac{1}{2}g_E(X,Y)|H_3|_E^2\right)\\
&\quad + \frac{1}{2}e^{2\phi}\left(G_1(X)G_1(Y)-\frac{1}{2}g_E(X,Y)|G_1|_E^2\right)\\
&\quad +\frac{1}{2}e^{\phi}\left(\langle i_{X}\tG{3},i_{Y}\tG{3}\rangle_E-\frac{1}{2}g_E(X,Y)|\tG{3}|_E^2\right)\\
&\quad +\frac{1}{4}\langle i_{X}\tG{5},i_{Y}\tG{5}\rangle_E
\end{align*}
\end{greybox}
or, equivalently, the {\bf Einstein equation in Ricci form:}
\begin{greybox}
\begin{align*}
\mathrm{Ric}_E(X,Y)&=\frac{1}{2}d\phi(X)d\phi(Y)\\
&\quad +\frac{1}{2}e^{-\phi}\left(\langle i_XH_3,i_YH_3\rangle_E-\frac{1}{4}g_E(X,Y)|H_3|_E^2\right)\\
&\quad + \frac{1}{2}e^{2\phi}G_1(X)G_1(Y)\\
&\quad +\frac{1}{2}e^{\phi}\left(\langle i_X\tG{3},i_Y\tG{3}\rangle_E-\frac{1}{4}g_E(X,Y)|\tG{3}|_E^2\right)\\
&\quad +\frac{1}{4}\langle i_X\tG{5},i_Y\tG{5}\rangle_E.
\end{align*}
\end{greybox}
\item {\bf Dilaton equation:}
\begin{greybox}
\begin{equation*}
\Delta_E\phi=-\frac{1}{2}e^{-\phi}|H_3|_E^2+e^{2\phi}|G_1|_E^2+\frac{1}{2}e^{\phi}|\tG{3}|^2_E.
\end{equation*}
\end{greybox}
\item {\bf Maxwell equations:}
\begin{greybox}
\begin{align*}
d\left(e^{2\phi}*_EG_1\right)&=-e^{\phi}H_3\wedge *_E\tG{3}\\
d*_E\tG{5}&=H_3\wedge \tG{3}\\
d\left(e^{\phi}*_E\tG{3}\right)&=-H_3\wedge \tG{5}\\
d\left(e^{-\phi}*_EH_3\right)&=\tG{3}\wedge\tG{5}+e^{\phi}G_1\wedge*_E\tG{3}.
\end{align*}
\end{greybox}
\item We add the {\bf self-duality equation:}
\begin{greybox}
\begin{equation*}
*\tG{5}=\tG{5}.
\end{equation*}
\end{greybox}
\item The {\bf Bianchi identities} stay the same:
\begin{greybox}
\begin{align*}
dH_3&=0\\
dG_1&=0\\
d\tG{3}&=H_3\wedge G_1\\
d\tG{5}&=H_3\wedge\tG{3}.
\end{align*}
\end{greybox}
\end{enumerate}
\end{thm}
\begin{rem}
These equations also appear in \cite{Fernand}.
\end{rem}
We only prove the dilaton equation, the remaining equations follow as before.
\begin{proof}
\begin{equation*}
\delta|d\phi|^2_E\mathrm{dvol}_{g_E}=d(\ldots)-2\delta\phi\Delta_E\phi\mathrm{dvol}_{g_E}.
\end{equation*}
Also
\begin{equation*}
\delta e^{\lambda\phi}=\lambda e^{\lambda\phi}\delta\phi.
\end{equation*}
This implies
\begin{equation*}
0=\Delta_E\phi+\frac{1}{2}e^{-\phi}|H_3|_E^2-e^{2\phi}|G_1|_E^2-\frac{1}{2}e^{\phi}|\tG{3}|^2_E
\end{equation*}
and hence the dilaton equation.
\end{proof}
\begin{cor}
The scalar curvature of a bosonic solution to the IIB Einstein equation in the Einstein frame is given by
\begin{greybox}
\begin{equation*}
R_E=\frac{1}{2}|d\phi|_E^2+\frac{1}{4}e^{-\phi}|H_3|_E^2+\frac{1}{2}e^{2\phi}|G_1|_E^2+\frac{1}{4}e^{\phi}|\tG{3}|_E^2.
\end{equation*}
\end{greybox}
\end{cor}

\subsection{Symmetric notation}

The field equations for the type IIB supergravity in the Einstein frame sometimes appear in a different form, for example, in \cite{Schwarz83}, \cite{Hull}, \cite{Schwarz95} and \cite{GMSW}, which we call the symmetric notation. We define the complex valued $1$-, $3$- and $0$-forms
\begin{greybox}
\begin{align*}
P&=\frac{i}{2}e^{\phi}G_1+\frac{1}{2}d\phi\\
G'_3&=ie^{\phi/2}(\tau H_3-G_3)\\
&=-ie^{\phi/2}\tilde{G_3}-e^{-\phi/2}H_3,
\end{align*}
where
\begin{equation*}
\tau=C_0+ie^{-\phi}
\end{equation*}
is the so-called axion-dilaton field.
\end{greybox}
We denote the complex conjugate of these fields by $^*$. The forms $G_1,d\phi,\tG{3}$ and $H_3$ can be obtained from the real and imaginary parts of these complex forms. We also define operators
\begin{greybox}
\begin{equation*}
D_q=\left(d-iqQ\wedge\right)
\end{equation*}
where $q\in\mathbb{R}$ and
\begin{equation*}
Q=-\frac{1}{2}e^{\phi}G_1.
\end{equation*}
\end{greybox}
\begin{thm}
With these fields the bosonic part of the type IIB action in the symmetric notation is given by
\begin{greybox}
\begin{equation*}
S=S^S_{NSRIIB}+S^S_{CSIIB}
\end{equation*}
with
\begin{align*}
S^S_{NSRIIB}&=\frac{1}{2\kappa_{10}^2}\int_M\mathrm{dvol}_{g_E}\left(R_E-\frac{1}{2(\mathrm{Im}\,\tau)^2}\langle d\tau,d\tau^*\rangle_E -\frac{1}{2}\mathcal{M}_{ij}\langle F_3^i,F_3^j\rangle_E-\frac{1}{4}|\tG{5}|^2_E\right)\\
S^S_{CSIIB}&=-\frac{1}{4\kappa_{10}^2}\int_MC_4\wedge F_3^1\wedge F_3^2,
\end{align*}
where
\begin{align*}
F_3&=\left(\begin{array}{c} H_3\\G_3\end{array}\right)\\
\mathcal{M}&=e^\phi\left(\begin{array}{cc}C_0^2+e^{-2\phi} & -C_0 \\ -C_0 & 1\end{array}\right).
\end{align*}
\end{greybox}
\end{thm}
\begin{proof}
We have
\begin{align*}
\frac{1}{2(\mathrm{Im}\,\tau)^2}\langle d\tau,d\tau^*\rangle_E&=\frac{1}{2}e^{2\phi}|G_1|^2+\frac{1}{2}|d\phi|^2\\
\frac{1}{2}\mathcal{M}_{ij}\langle F_3^i,F_3^j\rangle_E&=\frac{1}{2}e^{-\phi}|H_3|_E^2+\frac{1}{2}e^{\phi}|\tG{3}|_E^2.
\end{align*}
This implies the formula for the action.
\end{proof}
\begin{thm}
The IIB action in the symmetric notation is invariant under the $SL(2,\mathbb{R})$-action given by
\begin{greybox}
\begin{align*}
\mathcal{M}'&=\Lambda\mathcal{M}\Lambda^T\\
F'&=\left(\Lambda^T\right)^{-1}F\\
\tau'&=\frac{a\tau+b}{c\tau+d}\\
g'&=g\\
C_4'&=C_4,
\end{align*}
where
\begin{equation*}
\Lambda=\left(\begin{array}{cc} a& b\\ c & d\end{array}\right)\in SL(2,\mathbb{R}).
\end{equation*}
\end{greybox}
\end{thm}
\begin{proof}
It is easy to see that $\mathcal{M}_{ij}\langle F_3^i,F_3^j\rangle_E$ is invariant (under $GL(2,\mathbb{R})$) and that the Chern-Simons term is invariant under $SL(2,\mathbb{R})$. Moreover, we have
\begin{align*}
\mathrm{Im}\,\tau'&=\frac{e^{-\phi}}{|c\tau+d|^2}\\
d\tau'&=\frac{1}{(c\tau+d)^2}d\tau
\end{align*}
for $\Lambda\in SL(2,\mathbb{R})$. This implies that the term 
\begin{equation*}
\frac{1}{2(\mathrm{Im}\,\tau)^2}\langle d\tau,d\tau'^*\rangle_E
\end{equation*}
is invariant.
\end{proof}
\begin{thm}
The bosonic equations of motion for the fields in the symmetric notation are:
\begin{enumerate}
\item {\bf Einstein equation in Ricci form:}
\begin{greybox}
\begin{align*}
\mathrm{Ric}_E(X,Y)&=P(X)P^*(Y)+P(Y)P^*(X)\\
&\quad +\frac{1}{4}\left(\langle i_XG_3',i_YG_3'^*\rangle_E+\langle i_YG'_3,i_XG_3'^*\rangle_E\right)\\
&\quad-\frac{1}{8}g(X,Y)\langle G_3',G_3'^*\rangle_E\\
&\quad +\frac{1}{4}\langle i_X\tG{5},i_Y\tG{5}\rangle_E.
\end{align*}
\end{greybox}
\item {\bf Maxwell equations:}
\begin{greybox}
\begin{align*}
D_2*_EP&=-\frac{1}{4}G_3'\wedge *_EG_3'\\
D_1*_EG_3'&=P\wedge *_EG_3'^*-iG_3'\wedge \tG{5}.
\end{align*}
\end{greybox}
The second Maxwell equation is sometimes written in the equivalent form
\begin{greybox}
\begin{equation*}
D_1*_EG_3'=P\wedge *_EG_3'^*-i\tG{5}\wedge *G_3'.
\end{equation*}
\end{greybox}
\item We add the {\bf self-duality equation:}
\begin{greybox}
\begin{equation*}
*\tG{5}=\tG{5}.
\end{equation*}
\end{greybox}
\item The {\bf Bianchi identities} are:
\begin{greybox}
\begin{align*}
D_2P&=0\\
D_1G_3'&=-P\wedge G_3'^*\\
d\tG{5}&=\frac{i}{2}G_3'\wedge G_3'^*.
\end{align*}
\end{greybox}
\end{enumerate}
\end{thm}
\begin{proof}
To prove the Einstein equation we calculate:
\begin{align*}
P(X)P^*(Y)+P(Y)P^*(X)&=\frac{1}{2}d\phi(X)d\phi(Y)+\frac{1}{2}e^{2\phi}G_1(X)G_1(Y)\\
\frac{1}{4}\left(\langle i_XG_3',i_YG_3'^*\rangle_E+\langle i_YG'_3,i_XG_3'^*\rangle_E\right)&=\frac{1}{2}\left(e^\phi\langle i_X\tG{3},i_Y\tG{3}\rangle_E+e^{-\phi}\langle i_XH_3,i_YH_3\rangle_E\right)\\
-\frac{1}{8}g(X,Y)\langle G_3',G_3'^*\rangle_E&=-\frac{1}{8}g(X,Y)\left(e^{\phi}|\tG{3}|_E^2+e^{-\phi}|H_3|_E^2\right).
\end{align*}
This proves the Einstein equation with Theorem \ref{thm:Einstein IIB}. To prove the Maxwell equations we calculate:
\begin{align*}
D_2*_EP&=d*_EP+ie^{\phi}G_1\wedge*_EP\\
&=\frac{i}{2}d\left(e^{-\phi}e^{2\phi}*_EG_1\right)+\frac{1}{2}d*_Ed\phi-\frac{1}{2}e^{2\phi}G_1\wedge*_EG_1+\frac{i}{2}e^{\phi}G_1\wedge*_Ed\phi\\
&=-\frac{i}{2}e^{\phi}d\phi\wedge*_EG_1-\frac{i}{2}H_3\wedge *_E\tG{3}+\frac{1}{2}\Delta_E\phi-\frac{1}{2}e^{2\phi}|G_1|_E^2+\frac{i}{2}G_1\wedge*_Ed\phi\\
&=-\frac{i}{2}H_3\wedge *_E\tG{3}-\frac{1}{4}e^{-\phi}|H_3|_E^2+\frac{1}{4}e^{\phi}|\tG{3}|_E^2\\
&=-\frac{1}{4}\left(-ie^{\phi/2}\tilde{G_3}-e^{-\phi/2}H_3\right)\wedge*_E\left(-ie^{\phi/2}\tilde{G_3}-e^{-\phi/2}H_3\right)\\
&=-\frac{1}{4}G_3'\wedge*_EG_3',
\end{align*}
where we used the dilaton equation and the Maxwell equation for $G_1$. Similarly,
\begin{align*}
D_1*_EG_3'&=d*_EG_3'+\frac{i}{2}e^{\phi}G_1\wedge*_EG_3'\\
&=-id\left(e^{-\phi/2}e^{\phi}*_E\tG{3}\right)-d\left(e^{\phi/2}e^{-\phi}*_EH_3\right)\\
&\quad +\frac{1}{2}e^{3\phi/2}G_1\wedge*_E\tG{3}-\frac{i}{2}e^{\phi/2}G_1\wedge *_EH_3\\
&=\frac{i}{2}e^{\phi/2}d\phi\wedge*_E\tG{3}+ie^{-\phi/2} H_3\wedge\tG{5}\\
&\quad -\frac{1}{2}e^{-\phi/2}d\phi\wedge*_EH_3 -e^{\phi/2}\tG{3}\wedge\tG{5}-e^{3\phi/2}G_1\wedge*_E\tG{3}\\
&\quad +\frac{1}{2}e^{3\phi/2}G_1\wedge*_E\tG{3}-\frac{i}{2}e^{\phi/2}G_1\wedge *_EH_3\\
&=ie^{\phi/2}P\wedge*_E\tG{3}-e^{-\phi/2}P\wedge*_EH_3-iG_3'\wedge\tG{5}\\
&=P\wedge*_EG_3'^*-iG_3'\wedge\tG{5}\\
&=P\wedge*_EG_3'^*-iG_3'\wedge*_E\tG{5}\\
&=P\wedge*_EG_3'^*-i\tG{5}\wedge*_EG_3',
\end{align*}
where we used the Maxwell equations for $\tG{3}$ and $H_3$ and the self-duality of $\tG{5}$.

The Bianchi identities are easy to check: The first one is equivalent to $dG_1=0$, the real and imaginary part of the second one are equivalent to $dH_3=0$ and $d\tG{3}=H_3\wedge G_1$, and the last one is equivalent to $d\tG{5}=H_3\wedge \tG{3}$.

\end{proof}

\section{Dimensional reduction from M-theory to IIA supergravity}

The action and Killing spinor equations of type IIA supergravity theory in dimension ten can be obtained from the eleven-dimensional supergravity theory by a process called {\em dimensional reduction}. In this section we follow the exposition in \cite{BBS} (see also \cite{Fig} for the mathematical idea behind the construction).

Let $N$ be a ten-dimensional manifold and $M$ the eleven-dimensional manifold $N\times S^1$, where $S^1$ is the circle. Our aim is to relate the eleven-dimensional supergravity on $M$ to the type IIA supergravity on $N$. We assume that we have the following fields on $N$:
\begin{itemize}
\item A Lorentz metric $g_N$.
\item A smooth function $\phi$.
\item A $1$-form $C_1$, a $2$-form $B_2$ and a $3$-form $C_3$ with field strengths 
\begin{align*}
G_2&=dC_1\\
H_3&=dB_2\\
G_4&=dC_3\\
\tG{4}&=G_4-C_1\wedge H_3.
\end{align*}
\end{itemize}
We also choose a metric $g_{S^1}$ on $S^1$, invariant under $S^1$-rotations, with volume form $\alpha^{10}$. The fields on $N$ define via pull-back under the projection $\pi_N$ fields on $M$ that are invariant under the canonical $S^1$-action. Similarly the metric and volume form on $S^1$ define fields on $M$ via pull-back under the projection $\pi_{S^1}$. We write $\pi_{S^1}^*\alpha^{10}=\alpha^{10}$.
\begin{defn}
Let $X$ be a tangent vector on $N$. Then we denote by $\bar{X}$ the tangent vectors on $M$, tangent to the $N$-factors, mapping under the differential ${\pi_N}_*$ to $X$. Similary a tangent vector $S$ to $S^1$ defines tangent vectors $\bar{S}$ on $M$, tangent to the $S^1$-fibres and mapping under ${\pi_{S^1}}_*$ to $S$.
\end{defn}

\begin{defn}
We have on $N\times S^1$ two different metrics: The Lorentzian product metric
\begin{equation*}
g_N\oplus g_{S^1}
\end{equation*}
and another metric $g_M$, defined as follows:
\begin{align*}
g_M(\bar{X},\bar{Y})&=e^{-2\phi/3}g_N(X,Y)+e^{4\phi/3}C_1(X)C_1(Y)\\
g_M(\bar{X},\bar{S})&=-e^{4\phi/3}C_1(X)\alpha^{10}(S)\\
g_M(\bar{S},\bar{S})&=e^{4\phi/3}g_{S^1}(S,S).
\end{align*}
Here $X,Y$ are tangent vectors on $N$ and $S$ is a tangent vector of $S^1$. 
\end{defn}
\begin{rem}\label{rem: sign BBS}
In \cite{BBS} the same definitions are used where the sign in front of $C_1$ in the definitions of $\tG{4}$ and $g_M(\bar{X},\bar{S})$ is the opposite one. We use the sign here so that the sign in the definition of $\tG{4}$ is the one we used before in Section \ref{sect: sugra IIA}.
\end{rem}
We also define a $3$-form $C$ on $M$ by
\begin{equation*}
C=\pi_N^*C_3+\pi_N^*B_2\wedge\alpha^{10}.
\end{equation*}
The field strength is 
\begin{equation*}
G=dC=\pi_N^*G_4+\pi_N^*H_3\wedge\alpha^{10}.
\end{equation*}
Note that the fields $g_M$, $C$ and $G$ are invariant under the action of $S^1$ on $M=N\times S^1$. We want to calculate the action of eleven-dimensional supergravity for the fields $g_M$ and $C$ on $M$ in terms of the fields $g_N$, $\phi$, $C_1,B_2$ and $C_3$ on $N$. 

\subsection{The Chern-Simons term} 
\begin{prop}
The Chern-Simons Lagrangian can be calculated as
\begin{equation*}
C\wedge G\wedge G=\pi_N^*\left(3B_2\wedge G_4\wedge G_4-dL  \right)\wedge\alpha^{10},
\end{equation*}
where $dL$ denotes the differential of a certain $9$-form $L$ on $N$.
\end{prop}
\begin{proof}
We have 
\begin{equation*}
G\wedge G=\pi_N^*(G_4\wedge G_4)+2\pi_N^*(G_4\wedge H_3)\wedge \alpha^{10}
\end{equation*}
and
\begin{equation*}
C\wedge G\wedge G=\pi_N^*(C_3\wedge G_4\wedge G_4)+2\pi_N^*(C_3\wedge G_4\wedge H_3)\wedge\alpha^{10}+\pi_N^*(B_2\wedge G_4\wedge G_4)\wedge\alpha^{10}.
\end{equation*}
The first term on the right is the pull-back of an $11$-form on a $10$-manifold and vanishes. We write
\begin{align*}
C_3\wedge G_4\wedge H_3&=C_3\wedge G_4\wedge dB_2\\
&=-d(C_3\wedge G_4\wedge B_2)+G_4\wedge G_4\wedge B_2.
\end{align*}
This implies the claim.
\end{proof}
\begin{cor}
The Chern-Simons term in the action of eleven-dimensional supergravity on $M=N\times S^1$ is given by
\begin{equation*}
-\frac{1}{12\kappa_{11}^2}\int_MC\wedge G\wedge G=-\frac{1}{4\kappa_{10}^2}\int_NB_2\wedge G_4\wedge G_4.
\end{equation*}
\end{cor}

\subsection{The Einstein term}
We want to calculate the Levi-Civita connection and the curvature of the metric $g_M$. We use the notion of Riemannian submersions \cite{ONeill}. From now on we denote by $\bar{S}=e_{10}$ the vector field on $M$ along $S^1$ with
\begin{equation*}
\alpha^{10}(S)=1.
\end{equation*}
We have
\begin{equation*}
g_M(\bar{S},\bar{S})=e^{4\phi/3}.
\end{equation*}
\begin{defn} The vertical subspace $\mV_p\subset T_pM$ is the subspace tangent to the $S^1$-fibre in $M=N\times S^1$, spanned by $\bar{S}$. The horizontal subspace $\mH_p\subset T_pM$ is the $g_M$-orthogonal complement of the vertical subspace. If $W$ is a tangent vector on $M$, then we denote its projection onto the horizontal and vertical subspaces by $W^\mH$ and $W^\mV$.
\end{defn}
A direct calculation shows:
\begin{prop}\label{prop:horizontal projection} Let $X$ and $Y$ be tangent vectors on $N$.
\begin{enumerate}
\item The projection $\bar{X}^\mH$ of the vector $\bar{X}$ onto the horizontal subspace is given by
\begin{equation*}
\bar{X}^\mH=\bar{X}+C_1(X)\bar{S}.
\end{equation*}
\item For the horizontal projections $\bar{X}^\mH$ and $\bar{Y}^\mH$ we have
\begin{equation*}
g_M(\bar{X}^\mH,\bar{Y}^\mH)=e^{-2\phi/3}g_N(X,Y).
\end{equation*}
\item We have $\pi_{N*}\bar{X}^\mH=X$.
\end{enumerate}
\end{prop}
We thus have two structures on the manifold $M$:
\begin{itemize}
\item The tangent space to the fibres $S^1$ and the tangent space to the $N$-factors determined by the product $M=N\times S^1$, orthogonal with respect to the metric $g_N\oplus g_{S^1}$.
\item The tangent space to the fibres $S^1$ and the orthogonal complement $\mH$ with respect to the metric $g_M$.
\end{itemize}
\begin{defn}
We define a metric
\begin{equation*}
g_N^\phi=e^{-2\phi/3}g_N.
\end{equation*}
\end{defn}
\begin{cor}
The projection $\pi_N$ is an isometry of the horizontal space $\mH_p$ with the metric $g_M$ onto the tangent space $T_pN$ with the metric $g_N^\phi$. Hence $\pi_N$ is by definition a Riemannian submersion between $(M,g_M)$ and $(N,g_N^\phi)$.
\end{cor}
We now begin with the calculation of the Levi-Civita connection and the curvature of the metric $g_M$. 
\begin{defn}
We denote the Levi-Civita connection of $g_M$ by $\nabla^M$, of $g_N$ by $\nabla^N$ and of $g_N^\phi$ by $\nabla^{N\phi}$.
\end{defn}
\begin{lem}\label{lem:commutat H proj S}
We have
\begin{align*}
\left[\bar{X}^\mH,\bar{Y}^\mH\right]^\mV&=G_2(X,Y)\bar{S}\\
\left[\bar{X}^\mH,\bar{S}\right]&=0.
\end{align*}
\end{lem}
\begin{proof}
We can write
\begin{equation*}
\left[\bar{X}^\mH,\bar{Y}^\mH\right]^\mV=e^{-4\phi/3}g_M\left(\left[\bar{X}^\mH,\bar{Y}^\mH\right], \bar{S}\right)\bar{S}
\end{equation*}
and
\begin{align*}
g_M\left(\left[\bar{X}^\mH,\bar{Y}^\mH\right], \bar{S}\right)&=g_M\left(\overline{[X,Y]}^\mH-C_1([X,Y])\bar{S}+L_XC_1(Y)\bar{S}-L_YC_1(X)\bar{S},\bar{S}\right)\\
&=G_2(X,Y)e^{4\phi/3}.
\end{align*}
Hence the first claim follows. The second claim is clear.
\end{proof}
\begin{defn}
For a $1$-form $\tau$ on $N$ we define a vector field $\tau^\sharp$ on $N$ by 
\begin{equation*}
g_N(\tau^\sharp,Y)=\tau(Y).
\end{equation*}
This implies
\begin{equation*}
g_M\left(\overline{\tau^\sharp}^\mH,Y^\mH\right)=e^{-2\phi/3}\tau(Y).
\end{equation*}
\end{defn}
Using the Koszul formula we get:
\begin{prop}\label{prop:Koszul M conn} The Levi-Civita connection $\nabla^M$ is given by:
\begin{align*}
\nabla^M_{\bar{S}}\bar{S}&=-\frac{2}{3}e^{2\phi}\overline{\left(d\phi\right)^\sharp}^\mH\\
\nabla^M_{\bar{X}^\mH}\bar{S}&=\nabla^M_{\bar{S}}\bar{X}^\mH=-\frac{1}{2}e^{2\phi}\overline{(i_XG_2)^\sharp}^\mH+\frac{2}{3}d\phi(X)\bar{S}\\
\nabla^M_{\bar{X}^\mH}\bar{Y}^\mH&=\overline{\left(\nabla^{N\phi}_XY\right)}^\mH+\frac{1}{2}G_2(X,Y)\bar{S}.
\end{align*}
\end{prop}
It is useful for calculations to choose an orthonormal basis for $g_M$.
\begin{lem}\label{lem:ONB for M and N} Let $e_0,e_1,\ldots,e_9$ be an orthonormal basis of $T_xN$ with respect to $g_N$ (with $e_0$ timelike) and $e_{10}=S$ as before the vector field along $S^1$ with $\alpha^{10}(e_{10})=1$. Then the following defines an orthonormal basis in every point $(x,\theta)$ of $M$ with $\theta\in S^1$ with respect to the metric $g_M$:
\begin{align*}
e'_a&=e^{\phi/3}\bar{e}^\mH_a=e^{\phi/3}(\bar{e}_a+{C_1}_a\bar{e}_{10})\quad\forall 0\leq a\leq 9\\
e'_{10}&=e^{-2\phi/3}\bar{e}_{10}.
\end{align*}
This implies
\begin{equation*}
\mathrm{dvol}_{g_M}=e^{-8\phi/3}\mathrm{dvol}_{g_N}\wedge\alpha^{10}.
\end{equation*}
We also have
\begin{align*}
[e_a',e_{10}']&=-\frac{2}{3}e^{\phi/3}d\phi_ae_{10}'\\
[e_a',e_b']&=\frac{1}{3}e^{\phi/3}\left(d\phi_ae'_b-d\phi_be'_a\right)+e^{2\phi/3}\overline{[e_a,e_b]}^\mH+e^{4\phi/3}{G_2}_{ab}e'_{10}.
\end{align*}
\end{lem}
\begin{lem} The Levi-Civita connection of the metric $g_N^\phi$ is related to the Levi-Civita connection of $g_N$ by
\begin{equation*}
g_N\left(\nabla^{N\phi}_ae_b,e_c\right)=g_N\left(\nabla^N_ae_b,e_c\right)-\frac{1}{3}(d\phi_a\eta_{bc}+d\phi_b\eta_{ca}-d\phi_c\eta_{ab}).
\end{equation*}
\end{lem}
\begin{cor}\label{cor:LC connection M onb} In the basis vector fields $e_0',\ldots,e_{10}'$ the Levi-Civita connection $\nabla^M$ is given by:
\begin{align*}
\nabla^M_{10'}e'_{10}&=-\frac{2}{3}e^{\phi/3}d\phi^ce'_c\\
\nabla^M_{10'}e_a'&=-\frac{1}{2}e^{4\phi/3}{G_2}_a^{\,\,\,c}e'_c+\frac{2}{3}e^{\phi/3}d\phi_ae_{10}'\\
\nabla^M_{a'}e_{10}'&=-\frac{1}{2}e^{4\phi/3}{G_2}_a^{\,\,\,c}e'_c\\
\nabla^M_{a'}e_b'&=e^{\phi/3}g_N(\nabla^N_ae_b,e_c)\eta^{cd}e_d'-\frac{1}{3}e^{\phi/3}(d\phi_be_a'-d\phi^c\eta_{ab}e_c')\\
&\quad+\frac{1}{2}e^{4\phi/3}{G_2}_{ab}e_{10}'.
\end{align*}
\end{cor}
Defining the coefficients
\begin{align*}
\omega^M_{\mu\nu\rho}&=g_M(\nabla^M_{\mu'}e_{\nu}',e_{\rho}')\\
\omega^N_{\mu\nu\rho}&=g_N(\nabla^N_\mu e_\nu,e_\rho)
\end{align*}
we get:
\begin{cor}\label{cor:conn 1-forms M} The connection $1$-forms of the Levi-Civita connection $\nabla^M$ are given by:
\begin{align*}
\omega^M_{1010c}&=-\frac{2}{3}e^{\phi/3}d\phi_c\\
\omega^M_{10bc}&=-\frac{1}{2}e^{4\phi/3}{G_2}_{bc}\\
\omega^M_{a10c}&=-\frac{1}{2}e^{4\phi/3}{G_2}_{ac}\\
\omega^M_{abc}&=e^{\phi/3}\omega^N_{abc}-\frac{1}{3}e^{\phi/3}(d\phi_b\eta_{ac}-d\phi_c\eta_{ab}).
\end{align*}
\end{cor}
\begin{rem}
If we use the sign in \cite{BBS} mentioned in Remark \ref{rem: sign BBS} above, we have to change the sign in front of $C_1$ and $G_2$ in Proposition \ref{prop:horizontal projection}, Lemma \ref{lem:commutat H proj S}, Proposition \ref{prop:Koszul M conn}, Lemma \ref{lem:ONB for M and N}, Corollary \ref{cor:LC connection M onb} and Corollary \ref{cor:conn 1-forms M}.
\end{rem}
\begin{defn}
For vector fields $V,W$ on $M$ we set
\begin{equation*}
A_VW=\left(\nabla^M_{V^\mH}W^\mH\right)^\mV+\left(\nabla^M_{V^\mH}W^\mV\right)^\mH.
\end{equation*}
Then $A$ is a tensor, i.e.~function linear in its arguments.
\end{defn}
\begin{lem}
We have
\begin{align*}
A_{e_a'}e_b'&=\frac{1}{2}e^{4\phi/3}{G_2}_{ab}e_{10}'\\
A_{e_a'}e_{10}'&=-\frac{1}{2}e^{4\phi/3}{G_2}_a^{\,\,\,c}e_c'.
\end{align*}
\end{lem}
There is the following general formula for the curvature of $M$ (this is equation $\{2\}$ in \cite{ONeill}, where the curvature is defined with the opposite sign).
\begin{lem}Suppose $X,Y,Z,W$ are vector fields on $N$. Then
\begin{align*}
g_M(R^M(\bar{X},\bar{Y})\bar{Z},\bar{W})&=g^\phi_N(R^{N\phi}(X,Y)Z,W)+2g_M(A_{\bar{X}}\bar{Y},A_{\bar{Z}}\bar{W})\\
&\quad - g_M(A_{\bar{Y}}\bar{Z},A_{\bar{X}}\bar{W})-g_M(A_{\bar{Z}}\bar{X},A_{\bar{Y}}\bar{W}).
\end{align*}
\end{lem}
\begin{prop} We get
\begin{align*}
g_M(R^M(e_a',e_b')e_c',e_d')\eta^{ad}\eta^{bc}&=e^{4\phi/3}g_N^\phi(R^{N\phi}(e_a,e_b)e_c,e_d)\eta^{ad}\eta^{bc}-\frac{3}{2}e^{8\phi/3}|G_2|^2\\
&=R^{N\phi}-\frac{3}{2}e^{8\phi/3}|G_2|^2\\
&=e^{2\phi/3}\left(R^N+6\Delta\phi-8|d\phi|^2\right)-\frac{3}{2}e^{8\phi/3}|G_2|^2.
\end{align*}
\end{prop}
In the second equation $R^{N\phi}$ denotes the scalar curvature of $g_N^\phi$ and we used that the vectors
\begin{equation*}
e_a^\phi=e^{\phi/3}e_a
\end{equation*}
are an orthonormal basis with respect to $g_N^\phi$. We also used the formula for the scalar curvature of a conformally changed metric from Section \ref{sect:IIA Einstein frame}.

Using the identity
\begin{equation*}
g_M(\nabla_A\nabla_BC,D)=L_Ag_M(\nabla_BC,D)-g_M(\nabla_BC,\nabla_AD)
\end{equation*}
and
\begin{equation*}
\Delta\phi=L_ad\phi^a-\omega^{Na}_{\,\,\,\,\,\,\,\,\,ad}d\phi^d
\end{equation*}
we can calculate:
\begin{prop}
We have
\begin{align*}
g_M(\nabla_{a'}\nabla_{10'}e_{10}',e_{d}')\eta^{ad}&=e^{2\phi/3}\left(\frac{16}{9}|d\phi|^2-\frac{2}{3}\Delta\phi\right)\\
g_M(\nabla_{10'}\nabla_{a'}e_{10}',e'_{d})\eta^{ad}&=-\frac{1}{2}e^{8\phi/3}|G_2|^2\\
g_M(\nabla_{[a',10']}e_{10}',e_{d}')\eta^{ad}&=\frac{4}{9}e^{2\phi/3}|d\phi|^2
\end{align*}
and
\begin{equation*}
g_M(R(e_a',e_{10}')e_{10}',e_d')\eta^{ad}=e^{2\phi/3}\left(\frac{4}{3}|d\phi|^2-\frac{2}{3}\Delta\phi\right)+\frac{1}{2}e^{8\phi/3}|G_2|^2.
\end{equation*}
\end{prop}
\begin{defn}
We set
\begin{equation*}
\frac{1}{\kappa_{10}^2}=\frac{1}{\kappa_{11}^2}\int_{S^1}\alpha^{10}.
\end{equation*}
(our $\kappa_{10}$ is called $\kappa$ in \cite{BBS} and $\tilde{\kappa}_{10}$ in \cite{BLT}).
\end{defn}
\begin{cor}The Einstein-Hilbert Lagrangian of $g_M$ is given by
\begin{equation*}
R^M\mathrm{dvol}_{g_M}=\left(e^{-2\phi}\left(R^N+\frac{14}{3}\Delta\phi-\frac{16}{3}|d\phi|^2\right)-\frac{1}{2}|G_2|^2\right)\mathrm{dvol}_{g_N}\wedge\alpha^{10}.
\end{equation*}
The Einstein term on the eleven-dimensional manifold $M=N\times S^1$ is given by
\begin{equation*}
\frac{1}{2\kappa_{11}^2}\int_M\mathrm{dvol}_{g_M}R^M=\frac{1}{2\kappa_{10}^2}\int_N\mathrm{dvol}_{g_N}e^{-2\phi}\left(R^N+4|d\phi|^2\right)-\frac{1}{4\kappa_{10}^2}\int_N\mathrm{dvol}_{g_N}|G_2|^2.
\end{equation*}
\end{cor}
Here we used that
\begin{equation*}
e^{-2\phi}\frac{14}{3}\Delta\phi=d\left(e^{-2\phi}\frac{14}{3}d*\phi\right)+\frac{28}{3}e^{-2\phi}|d\phi|^2.
\end{equation*}
\begin{rem}
Setting $g_s=e^\phi$ for the string coupling and
\begin{equation*}
2\pi R_{11}=\frac{\kappa_{11}^2}{\kappa_{10}^2g_s^2}=\frac{1}{g_s^2}\int_{S^1}\alpha^{10}
\end{equation*}
we see that the Einstein-Hilbert parts of the Lagrangians are related by
\begin{equation*}
\int_{S^1}\mathrm{dvol}_{g_M}R^M=(2\pi R_{11})\mathrm{dvol}_{g_N}R^N+\ldots
\end{equation*}
In this sense $R_{11}$ is the radius of the circle $S^1$ that forms the eleventh dimension. The gravitational coupling constants are related to the Planck length $l_p$ in eleven dimensions and the string length $l_s$ in ten dimensions by
\begin{align*}
\kappa_{11}^2&=\frac{1}{4\pi}(2\pi l_p)^9\\
\kappa_{10}^2&=\frac{1}{4\pi}(2\pi l_s)^8
\end{align*}
and the form of the metric $g_M$ indicates that
\begin{equation*}
l_p=g_s^{1/3}l_s.
\end{equation*}
This implies
\begin{equation*}
R_{11}=g_sl_s=g_s^{2/3}l_p,
\end{equation*}
so that the radius $R_{11}$ measured in string units is proportional to the string coupling $g_s$. This is the idea of an eleventh dimension opening up and becoming large as the type IIA string theory becomes strongly coupled \cite{W}.
\end{rem}

\subsection{The Maxwell term}
The following is a consequence of Proposition \ref{prop:horizontal projection}.

\begin{prop} With respect to the orthonormal basis from Lemma \ref{lem:ONB for M and N} we have:
\begin{align*}
G(e'_a,e'_b,e'_c,e'_d)&=e^{4\phi/3}\tG{4}(e_a,e_b,e_c,e_d)\\
G(e'_a,e'_b,e'_c,e'_{10})&=e^{\phi/3}H_3(e_a,e_b,e_c),
\end{align*}
where $0\leq a,b,c\leq 9$. This implies
\begin{equation*}
|G|^2_{g_M}\mathrm{dvol}_{g_M}=\left(|\tG{4}|^2_{g_N}+e^{-2\phi}|H_3|^2_{g_N}\right)\mathrm{dvol}_{g_N}\wedge\alpha^{10}
\end{equation*}
and for the Maxwell term
\begin{equation*}
-\frac{1}{4\kappa_{11}^2}\int_M\mathrm{dvol}_{g_M}|G|^2=-\frac{1}{4\kappa_{10}^2}\int_N\mathrm{dvol}_{g_N}\left(|\tG{4}|^2+e^{-2\phi}|H_3|^2\right).
\end{equation*}
\end{prop}

\subsection{The complete action}
We collect our results.
\begin{thm}
With the fields on $N$ and $M=N\times S^1$ related as above, the action of eleven-dimensional supergravity on $M$
\begin{equation*}
S_{11}=\frac{1}{2\kappa_{11}^2}\int_M \left(R^M-\frac{1}{2}|G|^2\right)\mathrm{dvol}_{g_M}-\frac{1}{12\kappa_{11}^2}\int_MC\wedge G\wedge G
\end{equation*}
is equal to the action of type IIA supergravity on $N$
\begin{equation*}
S_{IIA}=S_{NS}+S_{RIIA}+S_{CSIIA}
\end{equation*}
with
\begin{align*}
S_{NS}&=\frac{1}{2\kappa_{10}^2}\int_N\mathrm{dvol}_{g_N}e^{-2\phi}\left(R^N+4|d\phi|^2-\frac{1}{2}|H_3|^2\right)\\
S_{RIIA}&=-\frac{1}{4\kappa_{10}^2}\int_N\mathrm{dvol}_{g_N}\left(|G_2|^2+|\tG{4}|^2\right)\\
S_{CSIIA}&=-\frac{1}{4\kappa_{10}^2}\int_N B_2\wedge G_4\wedge G_4.
\end{align*}
\end{thm}

\subsection{Spinors on the manifolds $N$ and $M$}

We assume from now on that $N$ is spin with associated complex spinor bundle $S^N$. Then $M=N\times S^1$ is also spin and $S^M=\pi_N^*S^N$ is a complex vector bundle on $M$. Let $\varep$ be a section of $S^N$ on $N$. Then this defines a section $\bar{\varep}$ of $S^M$ on $M$, which maps to $\varep$ under the fibrewise isomorphisms
\begin{equation*}
{\pi_N}_*\colon S^M_{(p,\theta)}\longrightarrow S^N_p,\quad p\in N,\theta\in S^1.
\end{equation*}
We define Clifford multiplication on $S^M$ as follows:
\begin{align*}
e_a'\cdot\bar{\varep}&=\overline{e_a\cdot\varep}\\
e_{10}'\cdot\bar{\varep}&=-\overline{\mathrm{dvol}_{g_N}\cdot\varep}.
\end{align*}
If we describe the spinor $\varep$ on $N$ locally by a map $\delta\colon U\rightarrow\mathbb{C}^{32}$ with respect to a trivialization of the $Spin(9,1)$-principal bundle, then we can describe it on $U\times S^1$ by the map 
\begin{equation*}
\bar{\delta}=\delta\circ\pi_N.
\end{equation*}
We have
\begin{align*}
\Gamma_a'\bar{\delta}&=\overline{\Gamma_a\delta}\\
\Gamma_{10}'\bar{\delta}&=\overline{\Gamma_{11}\delta},
\end{align*}
where $\Gamma_{11}$ is defined by
\begin{equation*}
\Gamma_{11}=\Gamma_0\Gamma_1\cdots\Gamma_9
\end{equation*}
as in Section \ref{sect:Weyl spinors}. Then $\Gamma_0'\Gamma_1'\cdots\Gamma_{10}'$ acts as $+1$ on $S^M$ (compare with Section \ref{sect:spinors in dim 10 and 11}).

\subsection{Killing spinor equations for type IIA supergravity}\label{chapt:dim reduct}
We consider the Killing spinor equation on $M$ in the form
\begin{equation*}
\nabla_X\varep+\frac{1}{24}\left(3G\cdot(X\cdot\varep)-X\cdot(G\cdot\varep)\right)=0,
\end{equation*}
as in Section \ref{sect:M-theory susy}.
\begin{prop}The terms in the Killing spinor equation for eleven-dimensional supergravity for $X=e_{10}'$ can be written as:
\begin{align*}
\nabla^M_{10'}\bar{\delta}&=\frac{1}{4}e^{4\phi/3}\overline{G_2\cdot\delta}-\frac{1}{3}e^{\phi/3}\overline{d\phi\cdot(\Gamma_{11}\delta)}\\
G\cdot\bar{\delta}&=e^{4\phi/3}\overline{\tG{4}\cdot\delta}+e^{\phi/3}\overline{H_3\cdot(\Gamma_{11}\delta)}\\
G\cdot(\Gamma_{10}'\cdot\bar{\delta})&=e^{4\phi/3}\overline{\tG{4}\cdot(\Gamma_{11}\delta)}+e^{\phi/3}\overline{H_3\cdot\delta}\\
\Gamma_{10}'\cdot(G\cdot{\bar{\delta}})&=e^{4\phi/3}\overline{\tG{4}\cdot(\Gamma_{11}\delta)}-e^{\phi/3}\overline{H_3\cdot\delta}.
\end{align*}
This implies
\begin{equation}\label{eqn:dilatino killing IIA}
\left(\frac{1}{4}e^{\phi}G_2-\frac{1}{3}d\phi\Gamma_{11}+\frac{1}{12}e^{\phi}\tG{4}\Gamma_{11}+\frac{1}{6}H_3\right)\cdot\delta=0.
\end{equation}
\end{prop}

\begin{prop}The terms in the Killing spinor equation for eleven-dimensional supergravity for $X=e_{a}'$ can be written as:
\begin{align*}
\nabla^M_{a'}\bar{\delta}&=e^{\phi/3}\overline{\nabla_a^N\delta}+\frac{1}{6}e^{\phi/3}\overline{d\phi_b\Gamma^{b}_{\,\,\,a}\delta}-\frac{1}{4}e^{4\phi/3}\overline{(i_aG_2)\cdot\Gamma_{11}\delta}\\
G\cdot(\Gamma_a'\cdot\bar{\delta})&=e^{4\phi/3}\overline{\tG{4}\cdot(\Gamma_a\cdot\delta)}+e^{\phi/3}\overline{H_3\cdot(\Gamma_{11}\Gamma_a\cdot\delta)}\\
\Gamma_{a}'\cdot(G\cdot{\bar{\delta}})&=e^{4\phi/3}\overline{\Gamma_a\cdot(\tG{4}\cdot\delta)}+e^{\phi/3}\overline{\Gamma_a\cdot(H_3\cdot\Gamma_{11}\delta)}\\
\end{align*}
This implies
\begin{align}\label{eqn first version killing IIA}
\nonumber 0&=\nabla_a^N\delta+\frac{1}{6}d\phi_b\Gamma^{b}_{\,\,\,a}\delta-\frac{1}{4}e^{\phi}(i_aG_2)\cdot\Gamma_{11}\delta\\
&\quad +\frac{1}{24}e^{\phi}\left(3\tG{4}\cdot(\Gamma_a\cdot\delta)-\Gamma_a\cdot(\tG{4}\cdot\delta)\right)\\
\nonumber&\quad-\frac{1}{24}\left(3H_3\cdot(\Gamma_a\cdot\Gamma_{11}\delta)+\Gamma_a\cdot(H_3\cdot\Gamma_{11}\delta)\right).
\end{align}
\end{prop}
Since
\begin{equation*}
\Gamma^b_{\,\,\,a}=-\Gamma_a\Gamma^b-\delta^b_a
\end{equation*}
we can write
\begin{equation*}
\frac{1}{6}d\phi_b\Gamma^{b}_{\,\,\,a}\delta=-\frac{1}{6}\Gamma_a\cdot(d\phi\cdot\delta)-\frac{1}{6}d\phi_a\delta.
\end{equation*}
With
\begin{equation*}
\tilde{\delta}=e^{-\phi/6}\delta
\end{equation*}
equation \eqref{eqn first version killing IIA} becomes
\begin{align}\label{eqn second version killing IIA}
\nonumber 0&=\nabla_a^N\tilde{\delta}-\frac{1}{6}\Gamma_a\cdot(d\phi\cdot\tilde{\delta})-\frac{1}{4}e^{\phi}(i_aG_2)\cdot\Gamma_{11}\tilde{\delta}\\
&\quad +\frac{1}{24}e^{\phi}\left(3\tG{4}\cdot(\Gamma_a\cdot\tilde{\delta})-\Gamma_a\cdot(\tG{4}\cdot\tilde{\delta})\right)\\
\nonumber&\quad-\frac{1}{24}\left(3H_3\cdot(\Gamma_a\cdot\Gamma_{11}\tilde{\delta})+\Gamma_a\cdot(H_3\cdot\Gamma_{11}\tilde{\delta})\right).
\end{align}
Adding to this equation 
\begin{equation*}
\frac{1}{2}e^{-\phi/6}\Gamma_a\cdot(\Gamma_{11}\cdot \text{equation \eqref{eqn:dilatino killing IIA}})
\end{equation*}
we get
\begin{equation*}
\nabla_a^N\tilde{\delta}-\frac{1}{4}(i_aH_3)\cdot\Gamma_{11}\tilde{\delta}+\frac{1}{8}e^\phi\tG{4}\cdot(\Gamma_a\cdot\tilde{\delta})+\frac{1}{8}e^{\phi}G_2\cdot(\Gamma_a\cdot\Gamma_{11}\tilde{\delta})=0.
\end{equation*}

\begin{cor}The dilatino and gravitino Killing spinor equation for type IIA supergravity are
\begin{equation*}
0=\left(\frac{1}{4}e^{\phi}G_2-\frac{1}{3}d\phi\Gamma_{11}+\frac{1}{12}e^{\phi}\tG{4}\Gamma_{11}+\frac{1}{6}H_3\right)\cdot\varep
\end{equation*}
and
\begin{equation*}
0=\nabla_X\varep-\frac{1}{4}(i_XH_3)\cdot \Gamma_{11}\varep+\frac{1}{8}e^\phi\left(\tG{4}\cdot(X\cdot\varep)+G_2\cdot(X\cdot\Gamma_{11}\varep)\right).
\end{equation*}
\end{cor}
\begin{rem}
Using the sign in \cite{BBS} mentioned in Remark \ref{rem: sign BBS} above has the effect of changing the sign of $G_2$ in both Killing spinor equations. Then our dilatino Killing spinor equation becomes the same as equations (8.46) and (8.50) in \cite{BBS}. Furthermore, our equation \eqref{eqn first version killing IIA} becomes their (8.47). However, there still remains a difference to the gravitino Killing spinor equation in their (8.51): With the sign from Remark \ref{rem: sign BBS} we get for the last term in our equation
\begin{equation*}
-\frac{1}{8}e^{\phi}G_2\cdot(X\cdot\Gamma_{11}\varep)=\frac{1}{4}e^{\phi}\left((i_XG_2)\cdot\Gamma_{11}\varep\right)-\frac{1}{8}e^{\phi}X\cdot(G_2\cdot\Gamma_{11}\varep).
\end{equation*}
However, they have
\begin{equation*}
-\frac{1}{4}e^{\phi}(X^\flat\wedge G_2)\cdot\Gamma_{11}\varep=\frac{1}{4}e^{\phi}\left((i_XG_2)\cdot\Gamma_{11}\varep\right)-\frac{1}{4}e^{\phi}X\cdot (G_2\cdot\Gamma_{11}\varep).
\end{equation*}
\end{rem}

\subsubsection*{Address}

\noindent Mathematical Institute, Ludwig Maximilians University Munich\\
         Theresienstr.~39, 80333 Munich, Germany
\vspace{0.2cm}

\noindent Institute for Geometry and Topology, University of Stuttgart\\
          Pfaffenwaldring 57, 70569 Stuttgart, Germany

\end{document}